\documentclass[a4paper,reqno]{amsart}

\usepackage{amssymb}
\usepackage{latexsym}
\usepackage{amsmath}
\usepackage{hyperref}

\def\sD{{\mathfrak D}}      
\def\sG{{\mathfrak G}}   \def\sH{{\mathfrak H}}   
      
\def\sM{{\mathfrak M}}

      \def\dC{{\mathbb C}}

      \def\dR{{\mathbb R}}

\def\cG{{\mathcal G}}   \def\cH{{\mathcal H}}

\def\cP{{\mathcal P}}      
      
\def\cV{{\mathcal V}}

\def\ran{{\text{\rm ran\,}}}
\def\dom{{\text{\rm dom\,}}}

 \usepackage[usestackEOL]{stackengine}
\usepackage{scalerel}
\def\moverline#1{\ThisStyle{%
  \setbox0=\hbox{$\SavedStyle#1$}%
  \stackengine{1.4\LMpt}{$\SavedStyle#1$}{\rule{\wd0}{.3\LMpt}}{0}{c}{F}{F}{S}%
}}
 \newcommand{\myoverline}[1]{\mkern 1mu\moverline{\mkern-1mu#1\mkern-1mu}\mkern 1mu}

\newtheorem{theorem}{Theorem}[section]
\newtheorem*{thm*}{Theorem}
\newtheorem{proposition}[theorem]{Proposition}
\newtheorem{corollary}[theorem]{Corollary}
\newtheorem{lemma}[theorem]{Lemma}

\theoremstyle{definition}
\newtheorem{definition}[theorem]{Definition}
\newtheorem{example}[theorem]{Example}

\newtheorem{remark}[theorem]{Remark}

\numberwithin{equation}{section}

\title[]{Boundary value problems for adjoint pairs of operators}

\author{Jussi Behrndt}

 \address{Institut f\"ur Angewandte Mathematik \\
 Technische Universit\"at Graz \\
 Steyrergasse 30\\
 A-8010 Graz\\
 Austria}
 \email{behrndt@tugraz.at}

\begin{document}

\begin{abstract}
The 
notion of quasi boundary triples and their Weyl functions from extension theory of symmetric operators 
is extended to the general framework of adjoint pairs
of operators under minimal conditions on the boundary maps. With the help of the corresponding
abstract Titchmarsh-Weyl $M$-functions sufficient conditions for the unique solvability of the related boundary value 
problems are obtained and the solutions are expressed via Krein-type resolvent formulae. The abstract theory developed in
this manuscript can be applied to a large class of elliptic differential operators.
\end{abstract}

\maketitle

\section{Introduction}
Boundary value problems for elliptic partial differential operators are often treated within the abstract framework of adjoint pairs of operators. This approach has its roots 
in the works of M.I. Visik \cite{V52}, M.S. Birman \cite{B62} and G. Grubb \cite{G68}, and has been further developed in the context of boundary triples 
and their Weyl functions for adjoint pairs of abstract operators in, e.g. \cite{BGW09,BHMNW09,BMNW08,M10}. The notion of (ordinary) boundary triples for
adjoint pairs of operators goes back to L.I. Vainerman \cite{V80} and the monograph \cite{LS83} by V.E.~Lyantse and O.G.~Storozh; the corresponding Weyl functions and Krein-type resolvent formulae were provided later by M.M. Malamud and V.I. Mogilevskii in \cite{MM97,MM99,MM02,M06}; see also \cite{BGHN17,HMM05,HMM13}. For the special case of 
symmetric operators and the spectral analysis of 
their self-adjoint extensions
the boundary triple technique is nowadays very well established \cite{BHS20,B76,BGP08,DHMS00,DM91,GG91,G09,K75,S12} and has been applied and extended in various directions. 
Among many generalizations of the notion of ordinary boundary triples for symmetric operators are the so-called quasi boundary triples, generalized boundary triples, and 
boundary relations for symmetric operators and relations from \cite{BL07,DHMS06,DM95}, see \cite{BL12,BLLR18,BMN17,BM14,BS19,DHM20,DHM22,DHMS09,DHMS12} for subsequent developments 
and in this context we also refer to \cite{AGW14,AB09,ACE23,AE12,AEKS14,A00,A12,AP17,BEHL18,BLL13,GM08,GM09,GM11,G08,LT77,LS22,LS23,MPS16,P04,P08,PR09,P07,P13,R07,W17,W13} for other closely related approaches and  typical 
applications. 

The main objective of this paper is to extend the notion of quasi boundary triples and their Weyl functions for symmetric operators from \cite{BL07,BL12,BLLR18,BM14} to the general framework of adjoint pairs
of operators, and to develop the abstract theory around this concept;
in particular, the aim is to provide sufficient conditions for boundary parameters and boundary mappings to induce closed extensions with nonempty resolvent sets, and to describe their 
resolvents via Krein-type resolvent formulae. Here we shall work under minimal assumptions on the boundary operators in the triple, that is, we require an abstract version of Green's second identity (G),
a weaker density condition (D) on the range of the boundary mappings than usual, and a certain maximality condition (M). In our results we shall always state explicitely which assumptions
(G), (D), or (M) are needed for the actual statement. 

Let us briefly explain and motivate our approach 
and the main difference to the concept of ordinary boundary triples for 
adjoint pairs of operators. For this, consider two densely defined closed operators
$S$ and $\widetilde S$ in a Hilbert space $\sH$ such that $\widetilde S\subset S^*$ (or, equivalently $S\subset\widetilde S^*$).
Assume that the operators $T$ and $\widetilde T$ are cores of $S$ and $\widetilde S$, respectively, that is, their closures coincide with
$S^*$ and $\widetilde S^*$. The key feature in our theory is the assumption that there exists an auxiliary (boundary) 
Hilbert space $\cG$ and boundary mappings $\Gamma_0,\Gamma_1:\dom T\rightarrow\cG$ and $\widetilde\Gamma_0,\widetilde\Gamma_1:\dom \widetilde T\rightarrow\cG$ such that an abstract Green's identity 
\begin{flalign*}
{\rm (G)}&\qquad\qquad\qquad (Tf,g)_\sH-(f,\widetilde Tg)_\sH=(\Gamma_1 f,\Gamma_0 g)_\cG-(\Gamma_0 f,\widetilde\Gamma_1 g)_\cG&
\end{flalign*}
holds for all $f\in\dom T$ and $g\in\dom\widetilde T$. We emphasize that (G) is not required on the (full) domain of the adjoint operators $S^*$ and $\widetilde S^*$ (as is the case for ordinary boundary triples) and that (G) typically does not 
admit an extension onto $\dom S^*$ and $\dom\widetilde S^*$. 
We are mainly interested in the case $\dim\cG=\infty$ (as otherwise $T=S^*$ and $\widetilde T=\widetilde S^*$, and hence ordinary boundary triples can be used).
In addition to Green's identity (G) a 
density condition (D) or (DD) and a maximality condition (M) in Definition~\ref{qbt} is often needed for a fruitful and functioning theory.
We will then study extensions of $\widetilde S$ and $S$ which are restrictions of $T$ and $\widetilde T$, respectively, 
of the form 
 \begin{equation*}
 \begin{split}
  A_B f&= Tf,\qquad \dom A_B=\bigl\{f\in\dom T: B\Gamma_1 f = \Gamma_0 f\bigr\},\\
  \widetilde A_{\widetilde B} g&= \widetilde Tg,\qquad \dom \widetilde A_{\widetilde B}=\bigl\{g\in\dom \widetilde T: \widetilde B\widetilde\Gamma_1 g 
  = \widetilde \Gamma_0 g\bigr\},
 \end{split}
 \end{equation*}
 where $B$ and $\widetilde B$ are linear operators in $\cG$. In the general abstract setting the goal is to show that
 $A_B$ and $\widetilde A_{\widetilde B}$ are closed operators with nonempty resolvent sets,
 as this ensures unique solvability or well-posedness of the abstract Robin-type boundary value problems 
 \begin{equation*}
  (T-\lambda)f=h,\quad B\Gamma_1 f=\Gamma_0f,\quad\text{or}\quad (\widetilde T-\mu)g=k,\quad \widetilde B\widetilde\Gamma_1 g
  =\widetilde \Gamma_0g,
 \end{equation*}
 whenever $\lambda\in\rho(A_B)$ and $\mu\in\rho(\widetilde A_{\widetilde B})$ and $h,k\in\sH$.
 After proving an abstract Birman-Schwinger principle (in a symmetrized form) 
 we find sufficient conditions on the boundary parameter, the mapping properties of the boundary maps,
 and the associated Weyl functions such that the resolvents $A_B$ and $\widetilde A_{\widetilde B}$ can be explicitly 
 computed in terms of the resolvent of an underlying fixed extension and a perturbation term in the boundary space $\cG$.
 We find it useful for reference purposes to summarize our results for general adjoint pairs in Appendix~\ref{app} in the special case that $S=\widetilde S$ is a densely defined closed symmetric operator. In this situation our results generalize those from \cite{BL07,BL12,BLLR18,BS19} in the sense that we impose a weaker density condition (D) than usual. 

The present paper stays on an abstract operator theory level and we have decided to postpone the diverse applications to 
future investigations and projects (with the small exceptions Example~\ref{exi1} and Example~\ref{exi2}, where a strongly elliptic system on a Lipschitz domain following 
\cite{M00} is discussed). We only indicate briefly as a motivation, that
in the situation of a second order elliptic differential expression  
(with smooth coefficients on a domain $\Omega\subset\dR^n$, $n\geq 2$, with a smooth boundary $\partial\Omega$) the adjoints 
$S^*$ and $\widetilde S^*$ coincide with the maximal operators associated to 
the differential expression and its formal adjoint in $L^2(\Omega)$, whereas $T$ and $\widetilde T$ can be chosen, e.g. as the operator realizations defined on 
the second order Sobolev space $H^2(\Omega)$; in this case a natural choice is $\Gamma_0=\widetilde\Gamma_0$ 
as the Dirichlet trace and $\Gamma_1,\widetilde\Gamma_1$ as the conormal derivatives mapping into the boundary space $L^2(\partial\Omega)$. It is clear that the operators $A_B$ and $\widetilde A_{\widetilde B}$ above
are Robin realizations of the elliptic differential expression and its formal adjoint. Besides this standard situation 
sketched here many other applications to differential operators can be explored.

\subsection*{Acknowledgements}  
The author is most grateful 
for the stimulating research
stay and the hospitality at the University of Auckland, where some
parts of this paper were written. The author also gratefully acknowledges financial support by the Austrian Science Fund (FWF): P 33568-N.
This publication is based upon work from COST Action CA 18232 MAT-DYN-NET, supported by COST (European Cooperation in Science and Technology), www.cost.eu.

%

\section{Quasi boundary triples for adjoint pairs of operators}\label{sec2}

Let throughout this section $S$ and $\widetilde S$ be densely defined closed operators in a separable Hilbert space $\sH$ such that
\begin{equation}\label{dpab}
 (Sf,g)=(f,\widetilde S g),\qquad f\in\dom S,\, g\in\dom \widetilde S,
\end{equation}
holds. Note that this is the same as requiring $\widetilde S\subset S^*$ or $S\subset \widetilde S^*$. 
We shall call a pair $\{S,\widetilde S\}$ of operators with this property an {\it adjoint pair} or {\it dual pair}.
In this manuscript we are mainly interested in the situation
\begin{equation}\label{deffi}
\dim\bigl(\dom S^*/\dom\widetilde S\bigr)= \dim\bigl(\dom \widetilde S^*/\dom S\bigr)=\infty,
\end{equation}
although our results do not formally require this condition.
Note that in the special case $S=\widetilde S$ the property \eqref{dpab} shows that $S$ is a symmetric operator and \eqref{deffi} means that at least one of 
the defect numbers of $S$ is infinite. In the following we shall work with operators $T\subset S^*$ and 
$\widetilde T\subset\widetilde S^*$ in $\sH$ which (are typically not closed and) satisfy 
\begin{equation*}
 \myoverline T=S^*\quad\text{and}\quad \myoverline{\widetilde T}=\widetilde S^*,
\end{equation*}
or, equivalently $T^*=S$ and $\widetilde T^*=\widetilde S$. In this situation we 
shall say that $T$ and $\widetilde T$ are {\it cores} of $S^*$ and $\widetilde S^*$, respectively.

Inspired by the notion of quasi boundary triples for symmetric operators in \cite{BL07,BL12} and typical applications in elliptic boundary value problems in the non-symmetric situation (see, e.g. \cite{M00}) we extend the definition of
boundary triples for adjoint pairs in the abstract setting from \cite{LS83,V80}, see also \cite{BGW09,BHMNW09,BMNW08,HMM13,MM97,MM99,MM02}. At the same time we also impose a slightly weaker density condition (D) than usual and it will turn out that this is sufficient for a functioning theory.

\begin{definition}\label{qbt}
Let $\{S,\widetilde S\}$ be an adjoint pair of operators in $\sH$ and assume 
$T$ and $\widetilde T$ are cores of $S^*$ and $\widetilde S^*$, respectively.
We shall consider {\em triples} of the form $\{\cG,(\Gamma_0,\Gamma_1),(\widetilde\Gamma_0,\widetilde\Gamma_1)\}$ for the adjoint pair $\{S,\widetilde S\}$, where
$\cG$ is a Hilbert space and
\begin{equation*}
 \Gamma_0,\Gamma_1:\dom T\rightarrow\cG,\qquad \widetilde\Gamma_0,\widetilde\Gamma_1:\dom \widetilde T\rightarrow\cG,
\end{equation*}
are linear mappings such that
\begin{itemize}
 \item [{\rm (G)}] the abstract {\em Green's identity}
\begin{equation*}
 (Tf,g)_\cH-(f,\widetilde Tg)_\cH=(\Gamma_1 f,\widetilde \Gamma_0 g)_\cG-(\Gamma_0 f,\widetilde\Gamma_1 g)_\cG
\end{equation*}
holds for all $f\in\dom T$ and $g\in\dom \widetilde T$,
\item [{\rm (D)}] the ranges of $\Gamma_0:\dom T\rightarrow\cG$ and $\widetilde\Gamma_0:\dom \widetilde T\rightarrow \cG$ are dense,
\item [{\rm (M)}] the operators $A_0:=T\upharpoonright\ker\Gamma_0$ and $\widetilde A_0:=\widetilde T\upharpoonright\ker\widetilde \Gamma_0$ satisfy
\begin{equation}\label{adjointi}
A_0^*=\widetilde A_0\quad\text{and}\quad \widetilde A_0^*=A_0. 
\end{equation}
\end{itemize}
If $\{\cG,(\Gamma_0,\Gamma_1),(\widetilde\Gamma_0,\widetilde\Gamma_1)\}$ is such that {\rm (G)},
\begin{itemize}
\item [{\rm (DD)}] the ranges of $(\Gamma_0,\Gamma_1)^\top:\dom T\rightarrow\cG\times \cG$ and $(\widetilde\Gamma_0,\widetilde\Gamma_1)^\top:\dom \widetilde T\rightarrow\cG\times \cG$ are dense,
\end{itemize}
and {\rm (M)} hold, then $\{\cG,(\Gamma_0,\Gamma_1),(\widetilde\Gamma_0,\widetilde\Gamma_1)\}$ is said to be a {\em quasi boundary triple} for the adjoint pair $\{S,\widetilde S\}$.
\end{definition}

In the following we shall formulate all results under minimal assumptions on the triple $\{\cG,(\Gamma_0,\Gamma_1),(\widetilde\Gamma_0,\widetilde\Gamma_1)\}$, tacitly assuming that
$\{S,\widetilde S\}$ is an adjoint pair of operators in $\sH$, and $T$ and $\widetilde T$ are cores of $S^*$ and $\widetilde S^*$, respectively. We will refer to condition (G) as 
abstract {\it Green's identity},
condition (D) and the stronger condition (DD) as {\it density conditions}, and (M) is a {\it maximality condition}. Note that condition (DD) implies that $\ran\Gamma_0$, $\ran\Gamma_1$, $\ran\widetilde\Gamma_0$, and 
$\ran\widetilde\Gamma_1$ are all dense in $\cG$ individually and, in particular, condition (DD) implies condition (D). Furthermore, it is clear from (M) that both restrictions 
$A_0=T\upharpoonright\ker\Gamma_0$ and $\widetilde A_0=\widetilde T\upharpoonright\ker\widetilde \Gamma_0$ are closed operators in $\sH$. 
Note that for $A_0$ closed the first condition $A_0^*=\widetilde A_0$ in \eqref{adjointi} already implies the second condition $\widetilde A_0^*=A_0$, and in the same 
 way $\widetilde A_0$ closed and $\widetilde A_0^*=A_0$ imply the first condition $A_0^*=\widetilde A_0$ in \eqref{adjointi}. 
Later in Section~\ref{gamsec} and 
Section~\ref{abssec} we will typically
assume that the resolvent sets of $\rho(A_0)$ and $\rho(\widetilde A_0)$ are nonempty; cf. Lemma~\ref{reslem}.

\begin{remark}\label{remmilein}
Note that if (G) holds, then for the operators $A_0=T\upharpoonright\ker\Gamma_0$ and $\widetilde A_0=\widetilde T\upharpoonright\ker\widetilde \Gamma_0$ 
one has
\begin{equation*}
 (A_0f,g)-(f,\widetilde A_0 g)=(Tf,g)-(f,\widetilde Tg)=(\Gamma_1 f,\widetilde\Gamma_0 g)-(\Gamma_0 f,\widetilde\Gamma_1 g)=0
\end{equation*}
for $f\in\dom A_0$ and $g\in\dom\widetilde A_0$, and hence the inclusions 
\begin{equation*}
\widetilde A_0\subset A_0^*\quad\text{and}\quad A_0\subset\widetilde A_0^*
\end{equation*}
hold without further assumptions;
thus condition (M) is only needed for the inclusions $\widetilde A_0\supset A_0^*$ and $A_0\supset\widetilde A_0^*$.
Furthermore, condition (M) implies that the cores $T\subset S^*$ and 
$\widetilde T\subset\widetilde S^*$ are 
also extensions of $\widetilde S$ and $S$, respectively, since
$A_0\subset T$, $\widetilde A_0\subset\widetilde T$, lead to
\begin{equation}\label{incels}
\widetilde S=\widetilde T^*\subset\widetilde A_0^*=A_0\subset T\quad\text{and}\quad 
S=T^*\subset A_0^*=\widetilde A_0\subset \widetilde T. 
\end{equation}
\end{remark}

We point out that the operators $T$ and $\widetilde T$ in Definition~\eqref{qbt} 
are not unique and may also coincide with $S^*$ and $\widetilde S^*$, respectively. 
However, in this special case $T=S^*$ and $\widetilde T=\widetilde S$ the situation simplifies and reduces to the notion of 
ordinary boundary triples for adjoint pairs; cf. Proposition~\ref{ordiprop} and, e.g. \cite{BMNW08,MM02,LS83,V80}.
It is not difficult to see that the mappings $\Gamma_0,\Gamma_1,\widetilde\Gamma_0,\widetilde\Gamma_1$ are not unique, e.g. if the triple
$\{\cG,(\Gamma_0,\Gamma_1),(\widetilde\Gamma_0,\widetilde\Gamma_1)\}$ satisfies (G) then also the triple $\{\cG,(\Gamma_1,-\Gamma_0),(\widetilde\Gamma_1,-\widetilde\Gamma_0)\}$ 
satisfies (G). Therefore, by imposing condition (D) for $\Gamma_1$ and $\widetilde \Gamma_1$ and by requiring that 
the operators  $A_1:=T\upharpoonright\ker\Gamma_1$ and $\widetilde A_1:=\widetilde T\upharpoonright\ker\widetilde \Gamma_1$ satisfy condition (M), that is,
\begin{equation}\label{adjointo}
A_1^*=\widetilde A_1\quad\text{and}\quad \widetilde A_1^*=A_1, 
\end{equation}
the triple $\{\cG,(\Gamma_1,-\Gamma_0),(\widetilde\Gamma_1,-\widetilde\Gamma_0)\}$  has the same properties (G), (D), and (M) as the original triple  
$\{\cG,(\Gamma_0,\Gamma_1),(\widetilde\Gamma_0,\widetilde\Gamma_1)\}$. In particular, if $\{\cG,(\Gamma_0,\Gamma_1),(\widetilde\Gamma_0,\widetilde\Gamma_1)\}$
is a quasi boundary triple and \eqref{adjointo} holds,
then $\{\cG,(\Gamma_1,-\Gamma_0),(\widetilde\Gamma_1,-\widetilde\Gamma_0)\}$ is also a quasi boundary triple. 

The next lemma shows that the density condition (DD) in Definition~\ref{qbt}
can be concluded from the surjectivity of the maps $\Gamma_0$ and $\widetilde\Gamma_0$.

\begin{lemma}\label{ddlem}
Assume that the triple $\{\cG,(\Gamma_0,\Gamma_1),(\widetilde\Gamma_0,\widetilde\Gamma_1)\}$ for the adjoint pair $\{S,\widetilde S\}$ satisfies {\rm (G)} and {\rm (M)}.
Then the following assertions hold.
\begin{itemize}
 \item [{\rm (i)}] If $\ran\Gamma_0$ is dense in $\cG$ and $\ran\widetilde\Gamma_0=\cG$, then $\ran(\Gamma_0,\Gamma_1)^\top$ is dense in $\cG\times\cG$;
 \item [{\rm (ii)}] If $\ran\widetilde\Gamma_0$ is dense in $\cG$ and $\ran\Gamma_0=\cG$, then $\ran(\widetilde\Gamma_0,\widetilde\Gamma_1)^\top$ is dense in $\cG\times \cG$.
\end{itemize}
In particular, if $\ran\Gamma_0=\ran\widetilde\Gamma_0=\cG$, then condition {\rm (DD)} in Definition~\ref{qbt} is satisfied and 
$\{\cG,(\Gamma_0,\Gamma_1),(\widetilde\Gamma_0,\widetilde\Gamma_1)\}$ is a quasi boundary triple for the adjoint pair $\{S,\widetilde S\}$.
\end{lemma}

\begin{proof}
(i) Assume that $(\varphi,\varphi')^\top\in\cG\times \cG$ is orthogonal to the range of $(\Gamma_0,\Gamma_1)^\top:\dom T\rightarrow\cG\times \cG$
and choose $g\in\dom\widetilde T$ such that $\widetilde\Gamma_0 g=\varphi'$. Then we have 
\begin{equation}\label{densli}
 0=(\Gamma_1 f,\varphi')-(\Gamma_0 f,-\varphi)=(\Gamma_1 f,\widetilde \Gamma_0 g)-(\Gamma_0 f,-\varphi)
\end{equation}
for all $f\in\dom T$
and hence the abstract Green's identity (G) becomes
\begin{equation*}
 (Tf,g)-(f,\widetilde Tg)=(\Gamma_1 f,\widetilde \Gamma_0 g)-(\Gamma_0 f,\widetilde\Gamma_1 g)=(\Gamma_0 f,-\varphi-\widetilde\Gamma_1 g).
\end{equation*}
In particular, for $f\in\ker\Gamma_0=\dom A_0$ we have 
\begin{equation*}
 (A_0f,g)-(f,\widetilde Tg)=(Tf,g)-(f,\widetilde Tg)=0
\end{equation*}
and therefore $g\in\dom A_0^*=\dom \widetilde A_0=\ker\widetilde\Gamma_0$, that is, $\varphi'=\widetilde\Gamma_0 g=0$. Now \eqref{densli}
reduces to $0=(\Gamma_0 f,\varphi)$ for all $f\in\dom T$ and as $\ran\Gamma_0$ is dense in $\cG$ we conclude $\varphi=0$. This shows that the range of
$(\Gamma_0,\Gamma_1)^\top$ is dense in $\cG\times\cG$.

(ii) can be proved in the same way as (i). It is also clear from (i) and (ii) that under the assumption $\ran\Gamma_0=\ran\widetilde\Gamma_0=\cG$ the ranges of both $(\Gamma_0,\Gamma_1)^\top$ and $(\widetilde\Gamma_0,\widetilde\Gamma_1)^\top$ 
are dense in $\cG\times\cG$, that is, condition (DD) in Definition~\ref{qbt} holds and hence 
$\{\cG,(\Gamma_0,\Gamma_1),(\widetilde\Gamma_0,\widetilde\Gamma_1)\}$ is a quasi boundary triple.
\end{proof}

We note that in the situation of Lemma~\ref{ddlem} with $\ran\Gamma_0=\ran\widetilde\Gamma_0=\cG$ the 
quasi boundary triple $\{\cG,(\Gamma_0,\Gamma_1),(\widetilde\Gamma_0,\widetilde\Gamma_1)\}$ can be regarded as a {\it generalized boundary triple} 
in the context of adjoint pairs; for the case of symmetric operators see \cite{DM95} and, e.g. \cite{BMN17,DHMS06,DHMS12}.

\begin{lemma}\label{lemmakers}
Assume that the triple $\{\cG,(\Gamma_0,\Gamma_1),(\widetilde\Gamma_0,\widetilde\Gamma_1)\}$ for the adjoint pair $\{S,\widetilde S\}$ satisfies {\rm (G)}, {\rm (D)}, and {\rm (M)}.
 Then
 \begin{equation}\label{aq}
  \dom S=\ker\widetilde\Gamma_0\cap\ker\widetilde\Gamma_1\quad\text{and}\quad
  \dom \widetilde S=\ker\Gamma_0\cap\ker\Gamma_1.
 \end{equation}
\end{lemma}

\begin{proof}
 We will verify the identity $\dom S=\ker\widetilde\Gamma_0\cap\ker\widetilde\Gamma_1$;
 the second identity in \eqref{aq} can be shown in the same way.
 For this, let us consider some fixed $g\in\dom S=\dom T^*$. From $A_0\subset T$ and 
 the maximality condition (M) we obtain 
 $T^*\subset A_0^*=\widetilde A_0\subset\widetilde T$ (see \eqref{incels}) and hence $g\in\dom\widetilde A_0=\ker\widetilde\Gamma_0\subset\dom \widetilde T$. The abstract Green's identity (G) yields
 \begin{equation*}
 \begin{split}
  0&=(Tf,g)-(f, T^*g)\\
  &=(Tf,g)-(f,\widetilde Tg)\\
  &=(\Gamma_1 f,\widetilde \Gamma_0 g)-(\Gamma_0 f,\widetilde\Gamma_1 g)\\
  &=-(\Gamma_0 f,\widetilde\Gamma_1 g)
 \end{split}
 \end{equation*}
 for all $f\in\dom T$ and as $\ran\Gamma_0$ is dense in $\cG$ by condition (D) it follows that also $g\in\ker\widetilde\Gamma_1$. Conversely, for $g\in \ker\widetilde\Gamma_0\cap\ker\widetilde\Gamma_1$ the abstract Green's identity 
 implies
 \begin{equation*}
  (Tf,g)-(f,\widetilde Tg)=(\Gamma_1 f,\widetilde \Gamma_0 g)-(\Gamma_0 f,\widetilde\Gamma_1 g)=0
 \end{equation*}
 for all $f\in\dom T$. This shows $g\in\dom T^*=\dom S$.
\end{proof}

\begin{lemma}
 Assume that the triple
 $\{\cG,(\Gamma_0,\Gamma_1),(\widetilde\Gamma_0,\widetilde\Gamma_1)\}$ for the adjoint pair $\{S,\widetilde S\}$ satisfies {\rm (G)} and {\rm (DD)}. Then the mappings 
 \begin{equation}\label{clos}
  \begin{pmatrix} \Gamma_0\\ \Gamma_1\end{pmatrix}:\dom T\rightarrow\cG\times \cG\quad 
  \text{and}\quad \begin{pmatrix}\widetilde\Gamma_0\\ \widetilde\Gamma_1\end{pmatrix}:\dom \widetilde T\rightarrow\cG\times \cG
 \end{equation}
are both closable with respect to the graph norm of $T$ and $\widetilde T$, respectively. 
In particular, the individual mappings $\Gamma_0,\Gamma_1:\dom T\rightarrow\cG$ and $\widetilde\Gamma_0,\widetilde\Gamma_1:\dom\widetilde T\rightarrow\cG$
are closable.
\end{lemma}

\begin{proof}
 Consider a sequence $f_n\in\dom T$ such that $f_n\rightarrow 0$ and $Tf_n\rightarrow 0$
 as $n\rightarrow\infty$ and assume that $\Gamma_0 f_n\rightarrow \varphi$ and 
 $\Gamma_1 f_n\rightarrow \varphi'$, $n\rightarrow\infty$, for some $\varphi,\varphi'\in\cG$. Using (G) it follows that 
 \begin{equation*}
  \begin{split}
   0&=\lim_{n\rightarrow\infty} \bigl((Tf_n,g)-(f_n,\widetilde Tg)\bigr)\\
    &=\lim_{n\rightarrow\infty} \bigl((\Gamma_1 f_n,\widetilde \Gamma_0 g)-(\Gamma_0 f_n,\widetilde \Gamma_1 g)\bigr)\\
    &=(\varphi',\widetilde \Gamma_0 g)-(\varphi,\widetilde \Gamma_1 g)
  \end{split}
 \end{equation*}
and as $\ran(\widetilde\Gamma_0,\widetilde\Gamma_1)^\top$ is dense in $\cG\times\cG$ by (DD) we conclude $\varphi=\varphi'=0$.
This proves that the first mapping in \eqref{clos} is closable and the same
argument applies to the second mapping in \eqref{clos}.
\end{proof}

\begin{proposition}\label{ordiprop}
 Let 
 $\{\cG,(\Gamma_0,\Gamma_1),(\widetilde\Gamma_0,\widetilde\Gamma_1)\}$ be a quasi boundary triple for the adjoint pair $\{S,\widetilde S\}$. Then
 the following are equivalent:
 \begin{itemize}
  \item [{\rm (i)}] $T=S^*$ and $\widetilde T=\widetilde S^*$
  \item [{\rm (ii)}] $\ran(\Gamma_0,\Gamma_1)^\top=\cG\times\cG$ and $\ran(\widetilde\Gamma_0,\widetilde\Gamma_1)^\top=\cG\times\cG$.
 \end{itemize}
\end{proposition}

\begin{proof}
 (i)$\,\Rightarrow\,$(ii): Since $\ran(\Gamma_0,\Gamma_1)^\top$ and $\ran(\widetilde\Gamma_0,\widetilde\Gamma_1)^\top$ are both dense in $\cG\times\cG$ by (DD) it remains to show that both ranges are closed. We will provide the argument for $\ran(\Gamma_0,\Gamma_1)^\top$; the same reasoning applies to $\ran(\widetilde\Gamma_0,\widetilde\Gamma_1)^\top$. Consider a sequence $(\Gamma_0 f_n,\Gamma_1 f_n)^\top$, where 
 $f_n\in\dom T$  and assume that 
 \begin{equation}\label{limmis}
  \Gamma_0 f_n\rightarrow \varphi\quad\text{and}\quad \Gamma_1 f_n\rightarrow\varphi'
 \end{equation}
as $n\rightarrow\infty$ for some $\varphi,\varphi'\in\cG$. Note first that by the inclusion $\widetilde S\subset T=S^*$ in \eqref{incels} there exists a closed subspace $\cV\subset\sH\times\sH$
such that $T=\widetilde S\oplus\cV$. Therefore,
since $\dom\widetilde S=\ker\Gamma_0\cap\ker\Gamma_1$ by Lemma~\ref{lemmakers} it is no restriction 
to assume that $f_n\in\dom T$ satisfy 
\begin{equation}\label{mmmk}
 \left(\begin{pmatrix} f_n \\ Tf_n\end{pmatrix},\begin{pmatrix} k \\ \widetilde Sk\end{pmatrix}\right)=0,\quad k\in\dom\widetilde S.
\end{equation}
Let $(h,h')^\top\in\sH\times\sH$ be arbitrary and observe that there exist $g\in\dom\widetilde S^*$ and $k\in\dom\widetilde S$ such that 
\begin{equation*}
 \begin{pmatrix}
  h \\ h'
 \end{pmatrix}
= \begin{pmatrix} -\widetilde S^* g\\ g \end{pmatrix} + \begin{pmatrix} k \\ \widetilde Sk\end{pmatrix}.
\end{equation*}
Using \eqref{mmmk}, $\widetilde S^*=\widetilde T$, and the abtract Green's identity (G) we compute 
\begin{equation*}
 \begin{split}
  \left(\begin{pmatrix} f_n \\ Tf_n\end{pmatrix},\begin{pmatrix} h \\ h'\end{pmatrix}\right)
  &=\left(\begin{pmatrix} f_n \\ Tf_n\end{pmatrix},\begin{pmatrix} -\widetilde S^* g\\ g \end{pmatrix} +\begin{pmatrix} k \\ \widetilde Sk\end{pmatrix}\right)\\
  &=(Tf_n,g)-(f_n,\widetilde S^*g)\\
  &=(Tf_n,g)-(f_n,\widetilde Tg)\\
  &=(\Gamma_1 f_n,\widetilde\Gamma_0 g)-(\Gamma_0 f_n,\widetilde\Gamma_1 g) \rightarrow (\varphi',\widetilde\Gamma_0 g)-(\varphi,\widetilde\Gamma_1 g)
 \end{split}
\end{equation*}
as $n\rightarrow\infty$, where \eqref{limmis} entered in the last step. It follows that $(f_n,T f_n)^\top$ is a weak Cauchy sequence in $\sH\times\sH$ and 
hence weakly bounded and
thus bounded. This implies that there exists a weakly convergent subsequence, again denoted by $(f_n,T f_n)^\top$ with weak limit 
$(f,f')^\top \in\myoverline T$. By assumption $T=S^*$ and hence $f'=Tf$. Now we conclude
\begin{equation*}
\begin{split}
 (\Gamma_1 f,\widetilde\Gamma_0 g)-(\Gamma_0f,\widetilde\Gamma_1 g)&=(Tf,g)-(f,\widetilde Tg)\\
 &=\lim_{n\rightarrow\infty} \bigl((Tf_n,g)-(f_n,\widetilde Tg)\bigr)\\
 &=\lim_{n\rightarrow\infty} \bigl((\Gamma_1 f_n,\widetilde\Gamma_0 g)-(\Gamma_0f_n,\widetilde\Gamma_1 g)\bigr)\\
 &=(\varphi',\widetilde\Gamma_0 g)-(\varphi,\widetilde\Gamma_1 g),
\end{split}
 \end{equation*}
that is, 
\begin{equation*}
 \left(\begin{pmatrix}\Gamma_1 f - \varphi' \\ \varphi-\Gamma_0 f \end{pmatrix},\begin{pmatrix}\widetilde\Gamma_0 g \\ \widetilde\Gamma_1 g \end{pmatrix}\right)=0
\end{equation*}
for all $g\in\dom \widetilde T$. As $\ran(\widetilde\Gamma_0,\widetilde\Gamma_1)^\top$ is dense in $\cG\times\cG$ by (DD) we obtain 
$\varphi=\Gamma_0 f$ and $\varphi'=\Gamma_1 f$, in particular, $(\varphi,\varphi')^\top\in\ran(\Gamma_0,\Gamma_1)^\top$ and therefore 
$\ran(\Gamma_0,\Gamma_1)^\top$ is closed.
\vskip 0.15cm
\noindent
 (ii)$\,\Rightarrow\,$(i): Since $T$ and $\widetilde T$ are cores of $S^*$ and $\widetilde S^*$, respectively, it suffices to verify that 
 $T$ and $\widetilde T$ are closed. We will provide the proof for $T$; the same argument can be used to show that $\widetilde T$ is closed. 
 Consider a sequence $f_n$ in $\dom T$ such that $f_n\rightarrow f$ and $Tf_n\rightarrow f'$ as $n\rightarrow\infty$ for some $f,f'\in\sH$.
 Let $(\psi,\psi')^\top\in\cG\times\cG$ and choose $g\in\dom\widetilde T$ such that $\widetilde\Gamma_0 g=\psi'$ and 
 $\widetilde\Gamma_1 g=-\psi$, which is possible by our assumptions. Using (G) we compute
 \begin{equation*}
  \begin{split}
   \left(\begin{pmatrix} \Gamma_0 f_n\\ \Gamma_1 f_n\end{pmatrix},\begin{pmatrix} \psi \\ \psi'\end{pmatrix} \right)
   &=(\Gamma_1 f_n,\widetilde\Gamma_0 g)-(\Gamma_0 f_n,\widetilde\Gamma_1 g)\\
   &=(T_nf,g)-(f_n,\widetilde Tg)\rightarrow (f',g)-(f,\widetilde Tg)
  \end{split}
 \end{equation*}
as $n\rightarrow\infty$. This shows that $(\Gamma_0 f_n, \Gamma_1 f_n)^\top$ is a weak Cauchy sequence in $\cG\times\cG$ and thus weakly bounded and hence bounded. Therefore, there exists a weakly convergent subsequence, again denoted by $(\Gamma_0 f_n, \Gamma_1 f_n)^\top$ 
with weak limit $(\varphi,\varphi')^\top\in\cG\times\cG$. By assumption there exists $h\in\dom T$ such that $\Gamma_0 h=\varphi$
and $\Gamma_1 h=\varphi'$. Now it follows for $g\in\dom\widetilde T$ that
\begin{equation*}
 \begin{split}
  (f',g)-(f,\widetilde Tg)
  &=\lim_{n\rightarrow\infty}\bigl((T f_n,g)-(f_n,\widetilde T g)\bigr)\\
  &=\lim_{n\rightarrow\infty}\bigl((\Gamma_1 f_n,\widetilde\Gamma_0 g)-(\Gamma_0 f_n,\widetilde\Gamma_1 g)\bigr)\\
  &=(\varphi',\widetilde\Gamma_0 g)-(\varphi,\widetilde\Gamma_1 g)\\
  &=(\Gamma_1 h,\widetilde\Gamma_0 g)-(\Gamma_0 h,\widetilde\Gamma_1 g)\\
  &=(Th,g)-(h,\widetilde Tg)
 \end{split}
\end{equation*}
and hence $(h-f,\widetilde Tg)=(Th-f',g)$ for all $g\in\dom\widetilde T$. This implies $h-f\in\dom\widetilde T^*$
and $\widetilde T^* (h-f)=Th-f'$. As $\widetilde T^*=\widetilde S\subset T$ by  \eqref{incels} and $h\in\dom T$ we conclude $f\in\dom T$ and $Tf=f'$.
We have shown that $T$ is closed.
\end{proof}

The next result is of a slightly different nature: it provides a method to verify that a pair of given operators $T$ and $\widetilde T$
form a core of the adjoints of certain (minimal) operators $S$ and $\widetilde S$, respectively. To emphasize this different point of view 
we shall denote the assumptions on the boundary maps here by (G'), (D') or (DD'), and (M').

\begin{theorem}\label{ratethm}
 Let $\sH$ and $\cG$ be Hilbert spaces and let $T$ and $\widetilde T$ be operators in $\sH$. Assume that 
 \begin{equation}\label{qbtinvthm}
  \Gamma_0,\Gamma_1:\dom T\rightarrow\cG\quad\text{and}\quad\widetilde\Gamma_0,\widetilde\Gamma_1:\dom \widetilde T\rightarrow\cG
 \end{equation}
 are linear mappings such that 
 \begin{itemize}
  \item [{\rm (G')}] the abstract Green's identity 
  \begin{equation*}
 (Tf,g)-(f,\widetilde Tg)=(\Gamma_1 f,\widetilde \Gamma_0 g)-(\Gamma_0 f,\widetilde\Gamma_1 g)
\end{equation*}
holds for all $f\in\dom T$ and $g\in\dom \widetilde T$,
  \item [{\rm (D')}] the ranges of $\Gamma_0:\dom T\rightarrow\cG$ and 
  $\widetilde\Gamma_0:\dom \widetilde T\rightarrow\cG$ are dense,
 \item [{\rm (M')}] the operators $A_0:=T\upharpoonright\ker\Gamma_0$ and $\widetilde A_0:=\widetilde T\upharpoonright\ker\widetilde \Gamma_0$ satisfy 
\begin{equation*}
A_0^*=\widetilde A_0\quad\text{and}\quad \widetilde A_0^*=A_0.
\end{equation*}
 \end{itemize}
If, in addition, $\ker\widetilde\Gamma_0\cap\ker\widetilde\Gamma_1$ and $\ker\Gamma_0\cap\ker\Gamma_1$ are dense in $\sH$, then 
the operators 
\begin{equation*}
\begin{split}
 Sf:= \widetilde Tf,&\quad f\in \dom S=\ker\widetilde\Gamma_0\cap\ker\widetilde\Gamma_1,\\
 \widetilde Sg:=  Tg,&\quad g\in\dom \widetilde S=\ker\Gamma_0\cap\ker\Gamma_1,
\end{split}
\end{equation*}
are closed and form an adjoint pair such that $T$ and $\widetilde T$ are cores of $S^*$ and $\widetilde S^*$, respectively. Furthermore, if the mappings in \eqref{qbtinvthm} satisfy the conditions
{\rm (G')},
\begin{itemize}
\item [{\rm (DD')}] the ranges of $(\Gamma_0,\Gamma_1)^\top:\dom T\rightarrow\cG\times \cG$ and $(\widetilde\Gamma_0,\widetilde\Gamma_1)^\top:\dom \widetilde T\rightarrow\cG\times \cG$ are dense,
\end{itemize}
and {\rm (M')}, then
$\{\cG,(\Gamma_0,\Gamma_1),(\widetilde\Gamma_0,\widetilde\Gamma_1)\}$ is a quasi boundary triple for the adjoint pair $\{S,\widetilde S\}$.
\end{theorem}

\begin{proof}
Observe first that by (G') and the definition of $S$ and $\widetilde S$ we have 
$$
(Sf,g)-(f,\widetilde Sg)=(Tf,g)-(f,\widetilde Tg)=(\Gamma_1 f,\widetilde \Gamma_0 g)-(\Gamma_0 f,\widetilde\Gamma_1 g)=0
$$
for all $f\in\dom S$ and and $g\in\dom\widetilde S$, and hence $\{S,\widetilde S\}$ is an adjoint pair.
We will verify the identity
\begin{equation}\label{feinfein}
T^*=S
\end{equation}
and the same arguments can be used to prove the identity $\widetilde T^*=\widetilde S$.    
Note first that from $A_0\subset T$ and (M') it follows that $T^*\subset A_0^*=\widetilde A_0 \subset \widetilde T$.   
Therefore, if $g\in\dom T^*$, then $g\in\dom \widetilde A_0=\ker\widetilde\Gamma_0\subset\dom \widetilde T$ and (G') implies 
 \begin{equation*}
 \begin{split}
0&= (Tf,g)-(f,T^*g)\\
 &=(Tf,g)-(f,\widetilde Tg)\\
 &=(\Gamma_1 f,\widetilde \Gamma_0 g)-(\Gamma_0 f,\widetilde\Gamma_1 g)\\
 &=-(\Gamma_0 f,\widetilde\Gamma_1 g)
\end{split}
 \end{equation*}
for all $f\in\dom T$, and hence assumption (D') shows $\widetilde\Gamma_1 g=0$. 
Now it follows that $T^*g=\widetilde Tg$ and $g\in\ker\widetilde\Gamma_0\cap\ker\widetilde\Gamma_1$, that is, $T^*\subset S$.
For the reverse inclusion let $g\in\ker\widetilde\Gamma_0\cap\ker\widetilde\Gamma_1$. Then it follows from (G') that
\begin{equation*}
 (Tf,g)-(f,\widetilde Tg)=0
\end{equation*}
holds for all $f\in\dom T$. This implies $g\in\dom T^*$ and $T^* g=\widetilde Tg$ and hence we obtain $S\subset T^*$. 
We have shown \eqref{feinfein}. It is also clear from \eqref{feinfein} that the operator $S$ is closed and $\myoverline T=T^{**}=S^*$ shows
that $T$ is a core for $S^*$. In the same way $\widetilde T^*=\widetilde S$ implies that $\widetilde S$ is closed and that 
$\widetilde T$ is a core for $\widetilde S^*$. Finally, note that $\{\cG,(\Gamma_0,\Gamma_1),(\widetilde\Gamma_0,\widetilde\Gamma_1)\}$ is a quasi boundary triple for the adjoint pair 
$\{S,\widetilde S\}$ if the conditions (G'), (DD'), and (M') hold.
\end{proof}

We briefly illustrate the abstract theory developed in this section for the case of strongly elliptic systems on Lipschitz domains following the 
presentation in \cite{M00}.

\begin{example}\label{exi1}
Let $\Omega\subset\dR^n$, $n\geq 2$, be a (possibly) unbounded Lipschitz domain with outward unit normal $\nu$ and consider a linear second order partial differential expression
\begin{equation*}
 \cP=-\sum_{j,k=1}^n\partial_j A_{jk} \partial_k + \sum_{j=1}^n A_j\partial_j +A,
\end{equation*}
with matrix-valued coefficient functions $A_{jk},A_j,A\in L^\infty(\Omega,\dC^{m\times m})$ such that $A_{jk},A_j$, $j,k=1,\dots,n$, are Lipschitz continuous, and its
formal adjoint 
\begin{equation*}
 \widetilde\cP=-\sum_{j,k=1}^n\partial_j A_{kj}^* \partial_k - \sum_{j=1}^n \partial_j A_j^* +A^*.
\end{equation*}
We define the operators $T$ and $\widetilde T$ in $L^2(\Omega,\dC^{m\times m})$ by
\begin{equation*}
\begin{split}
 Tf=\cP f,&\qquad\dom T=\bigl\{f\in H^1(\Omega,\dC^{m\times m}): \cP f\in L^2(\Omega,\dC^{m\times m})\bigr\},\\
 \widetilde Tg=\widetilde\cP g,&\qquad\dom \widetilde T=\bigl\{g\in H^1(\Omega,\dC^{m\times m}): \widetilde\cP g\in L^2(\Omega,\dC^{m\times m})\bigr\}.
 \end{split}
\end{equation*}
Recall that the Dirichlet trace operator 
\begin{equation}\label{dt}
\tau_D:H^1(\Omega,\dC^{m\times m})\rightarrow H^{1/2}(\partial\Omega,\dC^{m\times m})
\end{equation}
is bounded and surjective. It follows from
the considerations in
\cite[Lemma 4.3 and Theorem 4.4]{M00} that the conormal derivatives
\begin{equation*}
 f\mapsto\sum_{j=1}^n\nu_j \tau_D (B_jf)\quad\text{and}\quad
 g\mapsto\sum_{j=1}^n\nu_j \tau_D (\widetilde B_jg),\quad f,g\in H^2(\Omega,\dC^{m\times m}),
 \end{equation*}
  where $B_j f=\sum_{k=1}^n A_{jk}\partial_k f$ and $\widetilde B_j g=\sum_{k=1}^n A_{kj}^*\partial_k g + A_j^*g$, 
  can be extended by continuity to mappings 
 \begin{equation*}
 \tau_N:\dom T\rightarrow H^{-1/2}(\partial\Omega,\dC^{m\times m})\quad\text{and}\quad\widetilde\tau_N:\dom \widetilde T\rightarrow H^{-1/2}(\partial\Omega,\dC^{m\times m})
\end{equation*}
such that 
\begin{equation}\label{gre}
 (Tf,g)-(f,\widetilde Tg)=\langle - \tau_N f,\tau_D g \rangle - \langle \tau_D f, - \widetilde \tau_N g \rangle
\end{equation}
holds for all $f\in\dom T$ and $g\in\dom \widetilde T$; here $\langle\cdot,\cdot\rangle$ denotes the usual dual pairing of $H^{1/2}(\partial\Omega,\dC^{m\times m})$
and $H^{-1/2}(\partial\Omega,\dC^{m\times m})$. Now choose isometric isomorphisms $\iota_\pm: H^{\pm 1/2}(\partial\Omega,\dC^{m\times m})\rightarrow L^2(\partial\Omega,\dC^{m\times m})$
that are compatible with this pairing, so that \eqref{gre} turns into 
\begin{equation}\label{gre111}
 (Tf,g)-(f,\widetilde Tg)= (- \iota_-\tau_N f,\iota_+\tau_D g ) - ( \iota_+\tau_D f, - \iota_-\widetilde \tau_N g ).
\end{equation}
Next define the operators $\widetilde S$ and $S$ as restrictions of $T$ and $\widetilde T$ onto $\ker\tau_D \cap\ker\tau_N$ and $\ker\tau_D \cap\ker\widetilde\tau_N$, respectively.
In a more explicit form we have
\begin{equation*}
\begin{split}
 \widetilde Sf&=\cP f,\\
 \dom \widetilde S&=\bigl\{f\in H^1(\Omega,\dC^{m\times m}): \cP f\in L^2(\Omega,\dC^{m\times m}), \tau_Df=0, \tau_N f=0\bigr\},
\end{split}
\end{equation*}
and
\begin{equation*}
\begin{split}
  Sg&= \widetilde \cP g,\\
 \dom S&=\bigl\{g\in H^1(\Omega,\dC^{m\times m}): \widetilde \cP g\in L^2(\Omega,\dC^{m\times m}), \tau_D g=0, \widetilde \tau_N g=0\bigr\}.
\end{split}
\end{equation*}
Let us now consider the triple $\{L^2(\partial\Omega,\dC^{m\times m}),(\iota_+\tau_D,-\iota_-\tau_N),(\iota_+\tau_D,-\iota_-\widetilde\tau_N)\}$. Observe that
(D) and (G) hold by \eqref{dt} and \eqref{gre111}. Furthermore, in the present situation we actually have 
$\ran \iota_+\tau_D=L^2(\partial\Omega,\dC^{m\times m})$. If, in addition, the resolvent set of the Dirichlet realizations 
\begin{equation*}
\begin{split}
 A_0f=\cP f,&\quad\dom A_0=\bigl\{f\in H^1(\Omega,\dC^{m\times m}): \cP f\in L^2(\Omega,\dC^{m\times m}),\,\tau_D f=0\bigr\},\\
 \widetilde A_0 g=\widetilde\cP g,&\quad\dom \widetilde A_0=\bigl\{g\in H^1(\Omega,\dC^{m\times m}): \widetilde\cP g\in L^2(\Omega,\dC^{m\times m}),\,\tau_D g=0\bigr\},
 \end{split}
\end{equation*}
is nonempty, then we also have $A_0^*=\widetilde A_0$ and $A_0=\widetilde A_0^*$ (see Lemma~\ref{reslem} below), and hence condition (M) holds; cf. \cite[Theorem 4.10]{M00} for the case of bounded Lipschitz domains. Now Theorem~\ref{ratethm} implies that $\{S,\widetilde S\}$ form an adjoint pair and that $\myoverline T=S^*$ and $\myoverline{\widetilde T}=\widetilde S^*$. 
Note that by Lemma~\ref{ddlem} the stronger density condition (DD) holds and hence $\{L^2(\partial\Omega,\dC^{m\times m}),(\iota_+\tau_D,-\iota_-\tau_N),(\iota_+\tau_D,-\iota_-\widetilde\tau_N)\}$ is a quasi 
boundary triple for the adjoint pair $\{S,\widetilde S\}$.
\end{example}

\section{$\gamma$-fields and Weyl functions}\label{gamsec}

In this section we introduce the notion of $\gamma$-fields and Weyl functions following the ideas in \cite{BL07,DM95,DHMS06} in the setting of adjoint pairs; cf. \cite{MM97,MM99,MM02}.
In the following we consider an adjoint pair $\{S,\widetilde S\}$ in $\sH$, cores 
$T$ and $\widetilde T$ of $S^*$ and $\widetilde S^*$, respectively, and a  triple 
$\{\cG,(\Gamma_0,\Gamma_1),(\widetilde\Gamma_0,\widetilde\Gamma_1)\}$ as in Definition~\ref{qbt}. In addition, we assume that the operators $A_0=T\upharpoonright\ker\Gamma_0$ and 
$\widetilde A_0=\widetilde T\upharpoonright\ker\widetilde\Gamma_0$
have nonempty resolvent sets $\rho(A_0)$ and $\rho(\widetilde A_0)$, respectively. Note that $A_0^*=\widetilde A_0$ and $\widetilde A_0^*=A_0$ in condition (M) imply $\lambda\in\rho(A_0)$ 
if and only if $\myoverline\lambda\in\rho(\widetilde A_0)$.

We provide a simple criterion for condition (M) to hold in the case that $\rho(A_0)$ and $\widetilde\rho(A_0)$ are nonempty. 

\begin{lemma}\label{reslem}
Let  $\{\cG,(\Gamma_0,\Gamma_1),(\widetilde\Gamma_0,\widetilde\Gamma_1)\}$ be a triple as in Definition~\ref{qbt} that satisfies
{\rm (G)}. Let $A_0=T\upharpoonright\ker\Gamma_0$ and 
$\widetilde A_0=\widetilde T\upharpoonright\ker\widetilde\Gamma_0$, and assume that there exists $\lambda_0\in\dC$ such that $\lambda_0\in\rho(A_0)$ 
and $\myoverline\lambda_0\in\rho(\widetilde A_0)$. Then $A_0=\widetilde A_0^*$ and  $A_0^*=\widetilde A_0$ and, in particular, condition {\rm (M)} is satisfied.
\end{lemma}

\begin{proof}
Since {\rm (G)} holds we have $A_0\subset\widetilde A_0^*$ by Remark~\ref{remmilein}. For $\lambda_0$ as in the assumptions one also has $\lambda_0\in\rho(\widetilde A_0^*)$, and hence it follows from
$A_0-\lambda_0\subset\widetilde A_0^*-\lambda_0$ that $A_0=\widetilde A_0^*$ and  $A_0^*=\widetilde A_0$ (since $\widetilde A_0$ is closed).
\end{proof}

 In the following we shall introduce and collect some properties of the so-called $\gamma$-fields $\gamma,\widetilde\gamma$ and Weyl functions $M,\widetilde M$
corresponding to the triple $\{\cG,(\Gamma_0,\Gamma_1),(\widetilde\Gamma_0,\widetilde\Gamma_1)\}$. Recall first the direct sum decompositions 
\begin{equation*}
\begin{split}
 \dom T=\dom A_0\,\dot +\,\ker(T-\lambda)=\ker \Gamma_0\,\dot +\,\ker(T-\lambda),\qquad\lambda\in\rho(A_0),\\
 \dom \widetilde T=\dom \widetilde A_0\,\dot +\,\ker(\widetilde T-\mu)=\ker \widetilde\Gamma_0\,\dot +\,\ker(\widetilde T-\mu),\qquad\mu\in\rho(\widetilde A_0).
 \end{split}
 \end{equation*}
 
 \begin{definition}
 Let  $\{\cG,(\Gamma_0,\Gamma_1),(\widetilde\Gamma_0,\widetilde\Gamma_1)\}$ be a triple as in Definition~\ref{qbt} and assume 
 that the resolvent sets of $A_0=T\upharpoonright\ker\Gamma_0$ and $\widetilde A_0=\widetilde T\upharpoonright\ker\widetilde\Gamma_0$ are nonempty. 
\begin{itemize}
 \item [{\rm (i)}] The $\gamma$-fields $\gamma$ and $\widetilde\gamma$ associated with $\{\cG,(\Gamma_0,\Gamma_1),(\widetilde\Gamma_0,\widetilde\Gamma_1)\}$ are defined by 
 \begin{equation*}
 \begin{split}
  \gamma(\lambda)&=\bigl(\Gamma_0\upharpoonright\ker(T-\lambda)\bigr)^{-1},\qquad \lambda\in\rho(A_0),\\
  \widetilde\gamma(\mu)&=\bigl(\widetilde\Gamma_0\upharpoonright\ker(\widetilde T-\mu)\bigr)^{-1},\qquad \mu\in\rho(\widetilde A_0).
 \end{split}
 \end{equation*}
 \item [{\rm (ii)}] The Weyl functions $M$ and $\widetilde M$ associated with $\{\cG,(\Gamma_0,\Gamma_1),(\widetilde\Gamma_0,\widetilde\Gamma_1)\}$ are defined by 
 \begin{equation*}
 \begin{split}
  M(\lambda)&=\Gamma_1\bigl(\Gamma_0\upharpoonright\ker(T-\lambda)\bigr)^{-1}=\Gamma_1\gamma(\lambda),\qquad \lambda\in\rho(A_0),\\
  \widetilde M(\mu)&=\widetilde\Gamma_1\bigl(\widetilde\Gamma_0\upharpoonright\ker(\widetilde T-\mu)\bigr)^{-1}=\widetilde\Gamma_1\widetilde\gamma(\mu),\qquad \mu\in\rho(\widetilde A_0).
 \end{split}
 \end{equation*}
\end{itemize}
 \end{definition}

Observe that for $f_\lambda\in\ker(T-\lambda)$, $\lambda\in\rho(A_0)$, and $g_\mu\in \ker(\widetilde T-\mu)$, $\mu\in\rho(\widetilde A_0)$, one has
\begin{equation}\label{weyl123}
 M(\lambda)\Gamma_0 f_\lambda =\Gamma_1 f_\lambda\quad\text{and}\quad
 \widetilde M(\mu)\widetilde\Gamma_0 g_\mu =\widetilde\Gamma_1 g_\mu.
\end{equation}

In the next proposition we collect some properties of the $\gamma$-fields.
 
\begin{proposition}\label{gamprop}
Assume that the triple
 $\{\cG,(\Gamma_0,\Gamma_1),(\widetilde\Gamma_0,\widetilde\Gamma_1)\}$ satisfies {\rm (G)}, {\rm (D)}, {\rm (M)}, and that the resolvent set of 
 $A_0$ or, equivalently, the resolvent set of $\widetilde A_0$ is nonempty. 
Let $\gamma$ and $\widetilde\gamma$ be the $\gamma$-fields  associated with $\{\cG,(\Gamma_0,\Gamma_1),(\widetilde\Gamma_0,\widetilde\Gamma_1)\}$.
Then the following assertions hold for all $\lambda\in\rho(A_0)$ and $\mu\in\rho(\widetilde A_0)$.
\begin{itemize}
 \item [{\rm (i)}] $\gamma(\lambda)$ and $\widetilde\gamma(\mu)$ are bounded operators from $\cG$ into $\sH$ with dense domains
 $\dom\gamma(\lambda)=\ran\Gamma_0$ and $\dom\widetilde\gamma(\mu)=\ran\widetilde\Gamma_0$, and $\ran\gamma(\lambda)=\ker(T-\lambda)$
 and $\ran\widetilde\gamma(\mu)=\ker(\widetilde T-\mu)$;
 \item [{\rm (ii)}] for $\varphi\in\ran\Gamma_0$ and $\psi\in\ran\widetilde\Gamma_0$ the functions $\lambda\mapsto\gamma(\lambda)\varphi$ and $\mu\mapsto\widetilde\gamma(\mu)\psi$ are holomorphic 
 on $\rho(A_0)$ and $\rho(\widetilde A_0)$, respectively, and the relations
 \begin{equation}\label{jadoch}
  \begin{split}
   \gamma(\lambda)&=\bigl(I+(\lambda-\nu)(A_0-\lambda)^{-1}\bigr)\gamma(\nu),\qquad \lambda,\nu\in\rho(A_0),\\
   \widetilde\gamma(\mu)&=\bigl(I+(\mu-\omega)(\widetilde A_0-\mu)^{-1}\bigr)\widetilde\gamma(\omega),\qquad \mu,\omega\in\rho(\widetilde A_0),
  \end{split}
\end{equation}
hold;
 \item [{\rm (iii)}] $\gamma(\lambda)^*$ and $\widetilde\gamma(\mu)^*$ are everywhere defined bounded operators from $\sH$ to $\cG$
 and for all $f,g\in\sH$ one has 
 \begin{equation*}
  \gamma(\lambda)^*f=\widetilde\Gamma_1(\widetilde A_0-\myoverline\lambda)^{-1}f\quad\text{and}\quad
  \widetilde\gamma(\mu)^*g=\Gamma_1(A_0-\myoverline\mu)^{-1}g,
 \end{equation*}
 in particular, $\ran\gamma(\lambda)^*\subset\ran\widetilde\Gamma_1$ and $\ran\widetilde\gamma(\mu)^*\subset\ran\Gamma_1$.
\end{itemize}
\end{proposition}

\begin{proof}
(iii) Let $\lambda\in\rho(A_0)$, $\varphi\in\dom\gamma(\lambda)=\ran\Gamma_0$, and $f\in\sH$. Making use of the condition (M) we obtain $\myoverline\lambda\in\rho(\widetilde A_0)$ and hence there exists 
$g\in\dom \widetilde A_0=\ker\widetilde\Gamma_0$ such that $(\widetilde A_0-\myoverline\lambda)g =f$. By the definition of the $\gamma$-field we have 
$\Gamma_0\gamma(\lambda)\varphi=\varphi$ and now it follows from $\widetilde A_0\subset\widetilde T$ and the abstract Green's identity (G) that
 \begin{equation*}
  \begin{split}
   (\gamma(\lambda)\varphi,f)&=\bigl(\gamma(\lambda)\varphi,(\widetilde A_0-\myoverline\lambda)g\bigr)\\
   &=-\bigl((\lambda\gamma(\lambda)\varphi,g)-(\gamma(\lambda)\varphi,\widetilde A_0 g)\bigr)\\
   &=-\bigl(( T\gamma(\lambda)\varphi,g)-(\gamma(\lambda)\varphi,\widetilde T g)\bigr)\\
   &=-\bigl((\Gamma_1 \gamma(\lambda)\varphi,\widetilde\Gamma_0 g)-(\Gamma_0\gamma(\lambda)\varphi,\widetilde\Gamma_1 g)\bigr)\\
   &=\bigl(\varphi,\widetilde\Gamma_1 (\widetilde A_0-\myoverline\lambda)^{-1}f\bigr).
  \end{split}
 \end{equation*}
Since this identity holds for all $f\in\sH$ and $\varphi\in\ran\Gamma_0$ (the latter is a dense subspace of $\cG$ as we assume (D)) we conclude $\dom\gamma(\lambda)^*=\sH$ and $\gamma(\lambda)^*f=\widetilde\Gamma_1(\widetilde A_0-\myoverline\lambda)^{-1}f$ is valid for all $f\in\sH$. 
From the fact that the adjoint operator is automatically closed it follows that $\gamma(\lambda)^*$
is bounded. A similar computation leads to the identity $\widetilde\gamma(\mu)^*g=\Gamma_1(A_0-\myoverline\mu)^{-1}g$ for all $g\in\sH$ and implies that
$\widetilde\gamma(\mu)^*$ is also bounded and everywhere defined on $\sH$.

\noindent (i) It follows from (iii) that $\gamma(\lambda)^{**}=\myoverline{\gamma(\lambda)}$ and 
$\widetilde\gamma(\mu)^{**}=\myoverline{\widetilde\gamma(\mu)}$ are everywhere defined and bounded operators from $\cG$ to $\sH$ 
and hence also the operators $\gamma(\lambda)$ and $\widetilde\gamma(\mu)$ are bounded. The remaining assertions in (i) are immediate from
the definition of the $\gamma$-fields.

\noindent (ii) For $\lambda,\nu\in\rho(A_0)$ we use (iii) and compute 
\begin{equation*}
\begin{split}
 \gamma(\lambda)^*-\gamma(\nu)^*&=\widetilde\Gamma_1\bigl((\widetilde A_0-\myoverline\lambda)^{-1}-(\widetilde A_0-\myoverline\nu)^{-1}\bigr)\\
 &=(\myoverline\lambda-\myoverline\nu)\widetilde\Gamma_1(\widetilde A_0-\myoverline\nu)^{-1}(\widetilde A_0-\myoverline\lambda)^{-1}\\
 &=(\myoverline\lambda-\myoverline\nu)\gamma(\nu)^*(\widetilde A_0-\myoverline\lambda)^{-1}.
\end{split}
\end{equation*}
Taking adjoints and using (M) leads to
\begin{equation*}
 \myoverline{\gamma(\lambda)}-\myoverline{\gamma(\nu)}=(\lambda-\nu)(A_0-\lambda)^{-1}\myoverline{\gamma(\nu)}
\end{equation*}
and this implies the first identity in \eqref{jadoch}. The second identity in \eqref{jadoch} can be proved in the same way.
\end{proof}

Now we turn the properties of the Weyl functions.

\begin{proposition}\label{mprop}
Assume that the triple
 $\{\cG,(\Gamma_0,\Gamma_1),(\widetilde\Gamma_0,\widetilde\Gamma_1)\}$ satisfies {\rm (G)}, {\rm (D)}, {\rm (M)}, and that the resolvent set of 
 $A_0$ or, equivalently, the resolvent set of $\widetilde A_0$ is nonempty. 
Let $\gamma,\widetilde\gamma$ and $M,\widetilde M$ be the $\gamma$-fields and Weyl functions, respectively, associated with $\{\cG,(\Gamma_0,\Gamma_1),(\widetilde\Gamma_0,\widetilde\Gamma_1)\}$.
Then the following assertions hold for all $\lambda\in\rho(A_0)$ and $\mu\in\rho(\widetilde A_0)$:
\begin{itemize}
 \item [{\rm (i)}] $M(\lambda)$ and $\widetilde M(\mu)$ are operators in $\cG$ with dense domains
 $\dom M(\lambda)=\ran\Gamma_0$ and $\dom\widetilde M(\mu)=\ran\widetilde\Gamma_0$, and $\ran M(\lambda)\subset\ran \Gamma_1$
 and $\ran\widetilde M(\mu)\subset\ran\widetilde\Gamma_1$;
\item [{\rm (ii)}] $M(\lambda)\subset\widetilde M(\myoverline\lambda)^*$ and $\widetilde M(\myoverline\lambda)\subset M(\lambda)^*$ and one has the identities
\begin{equation}\label{yes3}
\begin{split}
 M(\lambda)-\widetilde M(\mu)^*&=(\lambda-\myoverline\mu)\widetilde\gamma(\mu)^*\gamma(\lambda),\\
 M(\lambda)^*-\widetilde M(\mu)&=(\myoverline\lambda-\mu)\gamma(\lambda)^*\widetilde\gamma(\mu);
 \end{split}
\end{equation}
\item [{\rm (iii)}]
The functions $\lambda\mapsto M(\lambda)$ and $\mu\mapsto\widetilde M(\mu)$ are holomorphic in the sense that they can be
written as the sum of the possibly unbounded closed operators $\widetilde M(\lambda_0)^*$ and $M(\mu_0)^*$, respectively, 
where $\lambda_0,\mu_0\in\rho(A_0)\cap\rho(\widetilde A_0)$ are fixed, and a bounded holomorphic operator
function:
\begin{equation*}
\begin{split} 
 M(\lambda)&=\widetilde M(\lambda_0)^*+\widetilde\gamma(\lambda_0)^*(\lambda-\myoverline\lambda_0)\bigl(I+(\lambda-\lambda_0)(A_0-\lambda)^{-1}\bigr)\gamma(\lambda_0),\\
 \widetilde M(\mu)&=M(\mu_0)^*+\gamma(\mu_0)^*(\mu-\myoverline\mu_0)\bigl(I+(\mu-\mu_0)(\widetilde A_0-\mu)^{-1}\bigr)\widetilde\gamma(\mu_0).
\end{split} 
\end{equation*}
\end{itemize}
\end{proposition}

\begin{proof}
(i) follows immediately from the definition of the Weyl functions $M$ and $\widetilde M$.
\\
\noindent 
(ii)
Let $\varphi_\lambda\in\ran\Gamma_0$ and $\psi_\mu\in\ran\widetilde\Gamma_0$ and pick $f_\lambda\in\ker(T-\lambda)$ and 
$g_\mu\in\ker(\widetilde T-\mu)$ such that $\Gamma_0 f_\lambda=\varphi_\lambda$ and $\widetilde\Gamma_0 g_\mu=\psi_\mu$.
Then we have $f_\lambda=\gamma(\lambda)\varphi_\lambda$ and $g_\mu=\widetilde\gamma(\mu)\psi_\mu$ and a straightforward 
computation leads to
\begin{equation}\label{doitplease}
 \begin{split}
(M(\lambda)\varphi_\lambda,\psi_\mu)-(\varphi_\lambda,\widetilde M(\mu)\psi_\mu)
&= (M(\lambda)\Gamma_0f_\lambda,\widetilde \Gamma_0 g_\mu)-(\Gamma_0f_\lambda,\widetilde M(\mu)\widetilde\Gamma_0 g_\mu)\\
&= (\Gamma_1f_\lambda,\widetilde \Gamma_0 g_\mu)-(\Gamma_0f_\lambda,\widetilde\Gamma_1 g_\mu)\\
&= (Tf_\lambda,g_\mu)-(f_\lambda,\widetilde Tg_\mu)\\
&= (\lambda f_\lambda,g_\mu) - (f_\lambda,\mu g_\mu)\\
&= \bigl((\lambda-\myoverline\mu)\gamma(\lambda)\varphi_\lambda,\widetilde\gamma(\mu)\psi_\mu\bigr).
 \end{split}
\end{equation}
For $\mu=\myoverline\lambda$ this leads to $(M(\lambda)\varphi_\lambda,\psi_{\myoverline\lambda})
=(\varphi_\lambda,\widetilde M(\myoverline\lambda)\psi_{\myoverline\lambda})$
and hence $\widetilde M(\myoverline\lambda)\subset M(\lambda)^*$ and $M(\lambda)\subset \widetilde M(\myoverline\lambda)^*$.
Furthermore, \eqref{doitplease} implies the identities \eqref{yes3}.
\\
\noindent
(iii) This is an immediate consequence of the formulas \eqref{doitplease} and Proposition~\ref{gamprop}~(ii).
\end{proof}

\section{Abstract boundary value problems}\label{abssec}

Next we introduce two families of operators in $\sH$ as restrictions of $T$ and $\widetilde T$ via abstract boundary conditions in $\cG$. 
 Let $\{S,\widetilde S\}$ be an adjoint pair of operators in $\sH$ and let 
 $\{\cG,(\Gamma_0,\Gamma_1),(\widetilde\Gamma_0,\widetilde\Gamma_1)\}$ be a triple for the adjoint pair $\{S,\widetilde S\}$ with linear mappings $\Gamma_0,\Gamma_1:\dom T\rightarrow \cG$ 
 and $\widetilde\Gamma_0,\widetilde\Gamma_1:\dom \widetilde T\rightarrow \cG$ 
 as in Definition~\ref{qbt}, where $T$ and $\widetilde T$ are cores of $S^*$ and $\widetilde S^*$, respectively.
 For linear operators $B$ and $\widetilde B$ in $\cG$ we define 
 \begin{equation}\label{abs}
 \begin{split}
  A_B f&= Tf,\qquad \dom A_B=\bigl\{f\in\dom T: B\Gamma_1 f = \Gamma_0 f\bigr\},\\
  \widetilde A_{\widetilde B} g&= \widetilde Tg,\qquad \dom \widetilde A_{\widetilde B}=\bigl\{g\in\dom \widetilde T: \widetilde B\widetilde\Gamma_1 g 
  = \widetilde \Gamma_0 g\bigr\}.
 \end{split}
 \end{equation}
Note that the operator $B$ may only be defined on a subspace of $\cG$ and hence $f\in\dom A_B$ means that $\Gamma_1 f\in\dom B$ and $B\Gamma_1 f=\Gamma_0 f$;
the boundary condition $\widetilde B\widetilde\Gamma_1 g 
  = \widetilde \Gamma_0 g$ is understood in the same way. The goal of this section is to derive conditions on the parameters $B$ or $\widetilde B$ and 
  the mapping properties of $\Gamma_0,\Gamma_1,\widetilde\Gamma_0,\widetilde\Gamma_1$ or the corresponding Weyl functions such that $A_B$ or $\widetilde A_{\widetilde B}$ have nonempty
  resolvent sets, that is, for $h,k\in\sH$ and $\lambda\in\rho(A_B)$ or $\mu\in\rho(\widetilde A_{\widetilde B})$ the boundary value problem
  \begin{equation*}
   (T-\lambda)f=h,\quad B\Gamma_1 f = \Gamma_0 f,\quad\text{or}\quad (\widetilde T-\mu)g=k,\quad \widetilde B\widetilde\Gamma_1 g = \widetilde\Gamma_0 g,
  \end{equation*}
admits a unique solution  $f=(A_B-\lambda)^{-1}h$ or $g=(\widetilde A_{\widetilde B}-\mu)^{-1}k$, which will be expressed in 
a resolvent formula involving the resolvent of $A_0,\widetilde A_0$ and a perturbation term consiting of the $\gamma$-fields, Weyl functions and parameters $B,\widetilde B$.
  
 We start with a simple preparatory lemma that only makes use of Green's identity  {\rm (G)}.

\begin{lemma}\label{bbblem}
Assume that the triple $\{\cG,(\Gamma_0,\Gamma_1),(\widetilde\Gamma_0,\widetilde\Gamma_1)\}$ for the adjoint pair $\{S,\widetilde S\}$ satisfies {\rm (G)}.
Let $B,B'$ be linear operators in $\cG$ and assume that
\begin{equation}\label{assub12}
 (B\varphi,\psi)=(\varphi, B'\psi)
\end{equation}
holds for all $\varphi\in\dom B$ and $\psi\in\dom B'$. Then the operators $A_B$ and $\widetilde A_{B'}$ in \eqref{abs} satisfy
\begin{equation}\label{abinc22}
 A_B\subset (\widetilde A_{B'})^* \quad\text{and}\quad \widetilde A_{B'}\subset (A_B)^*.
\end{equation}
In particular, if $B$ is densely defined then 
\begin{equation*}
 A_B\subset (\widetilde A_{B^*})^* \quad\text{and}\quad \widetilde A_{B^*}\subset (A_B)^*.
\end{equation*}
\end{lemma}

\begin{proof}
For $f\in\dom A_B\subset\dom T$ and $g\in\dom \widetilde A_{B'}\subset\dom\widetilde T$ it follows from Green's identity (G) that
 \begin{equation*}
  \begin{split}
   (A_Bf,g)-(f,\widetilde A_{B'}g)&=(Tf,g)-(f,\widetilde Tg)\\
   &=(\Gamma_1 f,\widetilde\Gamma_0 g)-(\Gamma_0 f,\widetilde\Gamma_1 g)\\
   &=(\Gamma_1 f, B'\widetilde\Gamma_1 g)-(B\Gamma_1 f,\widetilde\Gamma_1 g)\\
   &=0,
  \end{split}
 \end{equation*}
 where \eqref{assub12} was used in the last step.
This implies both inclusions in \eqref{abinc22}.
\end{proof}

In the next theorem we provide an abstract Birman-Schwinger principle in a symmetrized form for operators of the type
\begin{equation}\label{abs1212}
 \begin{split}
  A_{B_1B_2} f&= Tf,\qquad \dom A_{B_1B_2}=\bigl\{f\in\dom T: B_1B_2\Gamma_1 f = \Gamma_0 f\bigr\},\\
  \widetilde A_{\widetilde B_1\widetilde B_2} g&= \widetilde Tg,\qquad \dom \widetilde A_{\widetilde B_1\widetilde B_2}=\bigl\{g\in\dom \widetilde T: \widetilde B_1\widetilde B_2\widetilde\Gamma_1 g 
  = \widetilde \Gamma_0 g\bigr\},
 \end{split}
 \end{equation}
where $ B_1B_2$ and 
$\widetilde B_1 \widetilde B_2$ are (products of) linear operators  in $\cG$; cf. \eqref{abs}. The special case $B_2=B$, $B_1=I$, or $\widetilde B_2=\widetilde B$, $\widetilde B_1=I$,
in which \eqref{abs1212} reduces to \eqref{abs} will be mentioned separately.

\begin{theorem}\label{bsthm}
Consider a triple
 $\{\cG,(\Gamma_0,\Gamma_1),(\widetilde\Gamma_0,\widetilde\Gamma_1)\}$ as in Definition~\ref{qbt}, assume that 
 the resolvent sets of 
 $A_0$ and $\widetilde A_0$ are nonempty, and let  $M$ and $\widetilde M$ be the associated Weyl functions.
Then the following assertions hold for the operators $A_{B_1 B_2}$ and $\widetilde A_{\widetilde B_1\widetilde B_2}$ in \eqref{abs1212}, and all $\lambda\in\rho(A_0)$
and $\mu\in\rho(\widetilde A_0)$:
\begin{itemize}
 \item [{\rm (i)}] $\lambda\in\sigma_p(A_{B_1 B_2})$ if and only if $\ker(I-B_2M(\lambda)B_1)\not=\{0\}$, and in this case 
 \begin{equation}\label{kergam23}
  \ker(A_{B_1 B_2}-\lambda)=\bigl\{f_\lambda\in\ker(T-\lambda):\Gamma_0 f_\lambda= B_1\varphi,
  \varphi\in \ker(I-B_2M(\lambda)B_1)\bigr\}.
 \end{equation}
 \item [{\rm (ii)}] $\mu\in\sigma_p(\widetilde A_{\widetilde B_1\widetilde B_2})$ if and only if 
 $\ker(I-\widetilde B_2\widetilde M(\mu)\widetilde B_1)\not=\{0\}$, and in this case 
 \begin{equation*}
  \ker(\widetilde A_{\widetilde B_1\widetilde B_2}-\mu)
  =\bigl\{g_\mu\in\ker(\widetilde T-\mu):\widetilde \Gamma_0 g_\mu= \widetilde B_1\psi, 
  \psi\in \ker(I-\widetilde B_2\widetilde M(\mu)\widetilde B_1)\bigr\}.
 \end{equation*}
\end{itemize}
\end{theorem}

In the case $B_2=B$, $B_1=I$, or $\widetilde B_2=\widetilde B$, $\widetilde B_1=I$, we have the following corollary.

\begin{corollary}\label{bscor}
Consider a triple
 $\{\cG,(\Gamma_0,\Gamma_1),(\widetilde\Gamma_0,\widetilde\Gamma_1)\}$ as in Definition~\ref{qbt}, assume that 
 the resolvent sets of 
 $A_0$ and $\widetilde A_0$ are nonempty, and let  $M$ and $\widetilde M$ be the associated Weyl functions.
Then the following assertions hold for the operators $A_B$ and $\widetilde A_{\widetilde B}$ in \eqref{abs}, and all $\lambda\in\rho(A_0)$
and $\mu\in\rho(\widetilde A_0)$:
\begin{itemize}
 \item [{\rm (i)}] $\lambda\in\sigma_p(A_{B})$ if and only if $\ker(I-BM(\lambda))\not=\{0\}$, and in this case 
 \begin{equation*}
  \ker(A_B-\lambda)=\gamma(\lambda)\ker(I-BM(\lambda)).
 \end{equation*}
 \item [{\rm (ii)}] $\mu\in\sigma_p(\widetilde A_{\widetilde B})$ if and only if 
 $\ker(I-\widetilde B\widetilde M(\mu))\not=\{0\}$, and in this case 
 \begin{equation*}
  \ker(\widetilde A_{\widetilde B}-\mu)=\widetilde\gamma(\mu)\ker(I-\widetilde B\widetilde M(\mu)).
 \end{equation*}
\end{itemize}
\end{corollary}

\begin{proof}[Proof of Theorem~\ref{bsthm}]
 (i) Assume first that $\lambda\in\sigma_p(A_{B_1 B_2})$ and let $f_\lambda\not=0$ be a corresponding eigenfunction. Then $f_\lambda\in\ker(T-\lambda)$ and $\Gamma_0 f_\lambda\not=0$ as otherwise $f_\lambda\in\dom A_0\cap\ker(T-\lambda)=
 \ker(A_0-\lambda)=\{0\}$. Furthermore, $f_\lambda\in\dom A_{B_1B_2}$ satisfies the boundary condition
 $\Gamma_0 f_\lambda= B_1B_2\Gamma_1 f_\lambda$ and hence we obtain $B_2\Gamma_1 f_\lambda\not =0$ and  $B_1B_2\Gamma_1 f_\lambda\in\ran\Gamma_0=\dom M(\lambda)$. From the definition of the Weyl function (see \eqref{weyl123}) we conclude
 $$\Gamma_1 f_\lambda=M(\lambda)\Gamma_0 f_\lambda=M(\lambda)B_1B_2\Gamma_1 f_\lambda$$ and therefore
  \begin{equation*}
 \begin{split}
 0&=B_2\Gamma_1 f_\lambda-B_2M(\lambda)B_1B_2\Gamma_1 f_\lambda\\
&=\bigr(I-B_2M(\lambda)B_1\bigl)B_2\Gamma_1 f_\lambda
\end{split}
 \end{equation*}
that is, $B_2\Gamma_1 f_\lambda\in\ker(I-B_2M(\lambda)B_1)$ and, in particular, $\ker(I-B_2M(\lambda)B_1)\not=\{0\}$.

 For the converse let us fix some $\varphi\in\ker(I-B_2M(\lambda)B_1)$, $\varphi\not=0$, and note that 
 $\varphi=B_2M(\lambda)B_1\varphi$ implies, in particular, $B_1\varphi\in\dom M(\lambda)=\ran\Gamma_0$ and $M(\lambda)B_1\varphi\in\dom B_2$. Furthermore, we have 
 $B_1\varphi\not=0$ and 
 \begin{equation}\label{okgutgut}
  B_1\varphi=B_1B_2M(\lambda)B_1\varphi.
 \end{equation}

 Next, we choose 
 $f_\lambda\in\ker(T-\lambda)$ such that $\Gamma_0 f_\lambda=B_1\varphi$. Since 
 $\Gamma_1 f_\lambda=M(\lambda)\Gamma_0 f_\lambda=M(\lambda)B_1\varphi\in\dom B_2$ we have 
 $B_2\Gamma_1 f_\lambda= B_2 M(\lambda)B_1\varphi$ and hence \eqref{okgutgut} implies
 $$
 \Gamma_0 f_\lambda=B_1\varphi=B_1B_2\Gamma_1 f_\lambda.
 $$
 The identity \eqref{kergam23} follows from the above considerations.
 \end{proof}

\begin{theorem}\label{kreinthm}
Assume that the triple
 $\{\cG,(\Gamma_0,\Gamma_1),(\widetilde\Gamma_0,\widetilde\Gamma_1)\}$ satisfies {\rm (G)}, {\rm (D)}, {\rm (M)}, and that the resolvent set of 
 $A_0$ or, equivalently, the resolvent set of $\widetilde A_0$ is nonempty. 
Let $\gamma,\widetilde\gamma$ and $M,\widetilde M$ be the associated $\gamma$-fields and Weyl functions, respectively.
Then the following assertions hold for the operators $A_{B_1 B_2}$ and $\widetilde A_{\widetilde B_1\widetilde B_2}$ in \eqref{abs1212}, and all $\lambda\in\rho(A_0)$
and $\mu\in\rho(\widetilde A_0)$:
\begin{itemize}
\item [{\rm (i)}] If $\lambda\not\in\sigma_p(A_{B_1B_2})$ and $f\in\sH$ is such that
\begin{equation}\label{asskrein}
\widetilde\gamma(\myoverline\lambda)^*f\in\dom B_2\quad\text{and}\quad B_2\widetilde\gamma(\myoverline\lambda)^*f\in\ran(I-B_2M(\lambda)B_1),
\end{equation}
then $f\in\ran(A_{B_1B_2}-\lambda)$ and the  Krein-type formula 
\begin{equation}\label{kreini}
 (A_{B_1B_2}-\lambda)^{-1}f=(A_0-\lambda)^{-1}f+\gamma(\lambda)B_1\bigl(I-B_2M(\lambda)B_1\bigr)^{-1}B_2\widetilde\gamma(\myoverline\lambda)^*f
\end{equation}
holds.
In particular, if \eqref{asskrein} holds for all $f\in\sH$, that is, $
\ran \widetilde\gamma(\myoverline\lambda)^*\subset \dom B_2$ and $\ran B_2\widetilde\gamma(\myoverline\lambda)^*\subset \ran(I-B_2M(\lambda)B_1)$,
then $A_{B_1B_2}-\lambda$ is a bijective operator in $\sH$.
\item [{\rm (ii)}] If $\mu\not\in\sigma_p(\widetilde A_{\widetilde B_1\widetilde B_2})$ and $g\in\sH$ is such that
\begin{equation}\label{asskreinmu}
\gamma(\myoverline\mu)^*g\in\dom \widetilde B_2\quad\text{and}\quad \widetilde B_2\gamma(\myoverline\mu)^*g\in\ran(I-\widetilde B_2 \widetilde M(\mu)\widetilde B_1),
\end{equation}
then $g\in\ran(\widetilde A_{\widetilde B_1\widetilde B_2}-\mu)$ and the Krein-type formula
\begin{equation*}
 (\widetilde A_{\widetilde B_1\widetilde B_2}-\mu)^{-1}g=(\widetilde A_0-\mu)^{-1}g+\widetilde\gamma(\mu) \widetilde B_1\bigl(I-\widetilde B_2 \widetilde M(\mu)\widetilde B_1\bigr)^{-1}\widetilde B_2\gamma(\myoverline\mu)^*g
\end{equation*}
holds.
In particular, if \eqref{asskreinmu} holds for all $g\in\sH$, that is, $\ran\gamma(\myoverline\mu)^*\subset\dom \widetilde B_2$ and $\ran\widetilde B_2\gamma(\myoverline\mu)^*\in\ran(I-\widetilde B_2 \widetilde M(\mu)\widetilde B_1)$, then $\widetilde A_{\widetilde B_1\widetilde B_2}-\mu$ is a bijective operator in $\sH$.
\end{itemize}
\end{theorem}

In the same spirit as in Corollary~\ref{bscor} we formulate the special case $B_2=B$, $B_1=I$, or $\widetilde B_2=\widetilde B$, $\widetilde B_1=I$, separately
as a corollary.

\begin{corollary}
Assume that the triple
 $\{\cG,(\Gamma_0,\Gamma_1),(\widetilde\Gamma_0,\widetilde\Gamma_1)\}$ satisfies {\rm (G)}, {\rm (D)}, {\rm (M)}, and that the resolvent set of 
 $A_0$ or, equivalently, the resolvent set of $\widetilde A_0$ is nonempty. 
Let $\gamma,\widetilde\gamma$ and $M,\widetilde M$ be the associated $\gamma$-fields and Weyl functions, respectively.
Then the following assertions hold for the operators $A_B$ and $\widetilde A_{\widetilde B}$ in \eqref{abs}, and all $\lambda\in\rho(A_0)$
and $\mu\in\rho(\widetilde A_0)$:
\begin{itemize}
\item [{\rm (i)}] If $\lambda\not\in\sigma_p(A_B)$ and $f\in\sH$ is such that $\widetilde\gamma(\myoverline\lambda)^*f\in\dom B$ and $B\widetilde\gamma(\myoverline\lambda)^*f\in\ran(I-BM(\lambda))$,
then $f\in\ran(A_B-\lambda)$ and 
\begin{equation*}
 (A_B-\lambda)^{-1}f=(A_0-\lambda)^{-1}f+\gamma(\lambda)\bigl(I-BM(\lambda)\bigr)^{-1}B\widetilde\gamma(\myoverline\lambda)^*f.
\end{equation*}
\item [{\rm (ii)}] If $\mu\not\in\sigma_p(\widetilde A_{\widetilde B})$ and $g\in\sH$ is such that $\gamma(\myoverline\mu)^*g\in\dom \widetilde B$ and $\widetilde B\gamma(\myoverline\mu)^*g\in\ran(I-\widetilde B\widetilde M(\mu))$,
then $g\in\ran(\widetilde A_{\widetilde B}-\mu)$ and 
\begin{equation*}
 (\widetilde A_{\widetilde B}-\mu)^{-1}g=(\widetilde A_0-\mu)^{-1}g+\widetilde\gamma(\mu)\bigl(I-\widetilde B\widetilde M(\mu)\bigr)^{-1}\widetilde B\gamma(\myoverline\mu)^*g.
\end{equation*}
\end{itemize}
\end{corollary}

\begin{proof}[Proof of Theorem~\ref{kreinthm}]
 (i) Let $f\in\sH$ and observe that by the assumptions the element $h\in\sH$ given by
 \begin{equation}\label{haha}
  h=(A_0-\lambda)^{-1}f+\gamma(\lambda)B_1\bigl(I-B_2M(\lambda)B_1\bigr)^{-1}B_2\widetilde\gamma(\myoverline\lambda)^*f
 \end{equation}
 is well defined. In fact, the inverse $(I-B_2M(\lambda)B_1)^{-1}$ exists 
 as $\lambda\not\in\sigma_p(A_{B_1B_2})$ (see Theorem~\ref{bsthm}~(i)) and hence \eqref{asskrein} ensures that
 $$
 \bigl(I-B_2M(\lambda)B_1\bigr)^{-1}B_2\widetilde\gamma(\myoverline\lambda)^*f \in \dom\bigl(I-B_2M(\lambda)B_1\bigr)
 \subset \dom M(\lambda)B_1.
 $$
 Since $\dom  M(\lambda)=\dom\gamma(\lambda)$ this together with the definition of the $\gamma$-field implies 
 $$
 \gamma(\lambda)B_1\bigl(I-B_2M(\lambda)B_1\bigr)^{-1}B_2\widetilde\gamma(\myoverline\lambda)^*f\in\ker(T-\lambda)
 $$
 and, in particular, 
 all products on the right hand side are meaningful. We claim that $h\in\dom A_{B_1B_2}$. 
 First is clear that $h\in\dom T$ and $\dom A_0=\ker\Gamma_0$,
 the definition of $\gamma$ and $M$, and Proposition~\ref{gamprop}~(iii) imply
 \begin{equation}\label{id1}
 \begin{split}
  \Gamma_0 h&=\Gamma_0(A_0-\lambda)^{-1}f+\Gamma_0\gamma(\lambda)B_1\bigl(I-B_2M(\lambda)B_1\bigr)^{-1}B_2\widetilde\gamma(\myoverline\lambda)^*f\\
            &=B_1\bigl(I-B_2M(\lambda)B_1\bigr)^{-1}B_2\widetilde\gamma(\myoverline\lambda)^*f
          \end{split}
 \end{equation}
 and
\begin{equation}\label{id2}
\begin{split}  
  \Gamma_1 h&=\Gamma_1(A_0-\lambda)^{-1}f+\Gamma_1\gamma(\lambda)B_1\bigl(I-B_2M(\lambda)B_1\bigr)^{-1}B_2\widetilde\gamma(\myoverline\lambda)^*f\\
            &=\widetilde\gamma(\myoverline\lambda)^*f+M(\lambda)B_1\bigl(I-B_2M(\lambda)B_1\bigr)^{-1}B_2\widetilde\gamma(\myoverline\lambda)^*f.
\end{split}
 \end{equation}
Since $\widetilde\gamma(\myoverline\lambda)^*f\in\dom B_2$ and $(I-B_2M(\lambda)B_1)^{-1}B_2\widetilde\gamma(\myoverline\lambda)^*f\in\dom B_2M(\lambda)B_1$ we conclude from
\eqref{id2} that $\Gamma_1 h\in\dom B_2$ and 
\begin{equation*}
\begin{split}
 B_2\Gamma_1 h&=B_2\widetilde\gamma(\myoverline\lambda)^*f+B_2M(\lambda)B_1\bigl(I-B_2M(\lambda)B_1\bigr)^{-1}B_2\widetilde\gamma(\myoverline\lambda)^*f\\
 &=\bigl((I-B_2M(\lambda)B_1)+B_2M(\lambda)B_1\bigr) \bigl(I-B_2M(\lambda)B_1\bigr)^{-1}B_2\widetilde\gamma(\myoverline\lambda)^*f\\
 &=\bigl(I-B_2M(\lambda)B_1\bigr)^{-1}B_2\widetilde\gamma(\myoverline\lambda)^*f.
\end{split} 
\end{equation*}
The above identity shows that $B_2\Gamma_1 h\in\dom B_2M(\lambda)B_1\subset \dom B_1$ and hence
\begin{equation*}
 B_1 B_2\Gamma_1 h=B_1\bigl(I-B_2M(\lambda)B_1\bigr)^{-1}B_2\widetilde\gamma(\myoverline\lambda)^*f=\Gamma_0 h,
\end{equation*}
where \eqref{id1} was used in the last step. Therefore, $h\in\dom A_{B_1B_2}$ and from $A_{B_1B_2}\subset T$, $A_0\subset T$, and 
$\ran\gamma(\lambda)=\ker(T-\lambda)$ we conclude
\begin{equation*}
\begin{split}
 (A_{B_1B_2}-\lambda) h&=(T-\lambda) h \\
 &=(T-\lambda)(A_0-\lambda)^{-1}f+(T-\lambda)\gamma(\lambda)B_1\bigl(I-B_2M(\lambda)B_1\bigr)^{-1}B_2\widetilde\gamma(\myoverline\lambda)^*f\\
 &=f,
\end{split}
 \end{equation*}
that is, $f\in\ran(A_{B_1 B_2}-\lambda)$. Now the Krein type formula \eqref{kreini} follows from \eqref{haha} and 
$h=(A_{B_1B_2}-\lambda)^{-1}f$.
\end{proof}

Note that Theorem~\ref{kreinthm} provides criteria such that the extensions $A_{B_1 B_2}-\lambda$ or $\widetilde A_{\widetilde B_1\widetilde B_2}-\mu$ are bijective,
but here their inverses are not necessarily bounded and hence it does not follow automatically 
that $A_{B_1 B_2}$ or $\widetilde A_{\widetilde B_1\widetilde B_2}$ are closed extension with a nonempty resolvent set. The next theorem is our first 
result in this direction; it is an easy consequence of Theorem~\ref{kreinthm} and Lemma~\ref{bbblem}.

\begin{theorem}
Assume that the triple
 $\{\cG,(\Gamma_0,\Gamma_1),(\widetilde\Gamma_0,\widetilde\Gamma_1)\}$ satisfies {\rm (G)}, {\rm (D)}, {\rm (M)}, and that the resolvent set of 
 $A_0$ or, equivalently, the resolvent set of $\widetilde A_0$ is nonempty.
Let $B_1,B_2,B_1',B_2'$ be operators in $\cG$ that satisfy 
 $$
 (B_1B_2\varphi,\psi)=(\varphi,B_1'B_2'\psi ), \quad\varphi\in\dom B_1B_2,\,\,\psi\in \dom B_1'B_2',
 $$
and consider the operators $A_{B_1 B_2}$ and $\widetilde A_{B_1'B_2'}$ in \eqref{abs1212}. If there exists some $\lambda\in\rho(A_0)$ such that assumption \eqref{asskrein} holds for $\lambda$ and all $f\in\sH$ and 
assumption \eqref{asskreinmu} holds for $\myoverline\lambda$ and all $g\in\sH$, then 
\begin{equation}\label{nabitte}
A_{B_1 B_2} = (\widetilde A_{B_1'B_2'})^*
\end{equation}
is a closed operator in $\sH$ with nonempty resolvent set.
\end{theorem}

\begin{proof}
 It follows from Theorem~\ref{kreinthm} that both operators $A_{B_1 B_2}-\lambda$ and $\widetilde A_{B_1'B_2'}-\myoverline\lambda$ are bijective. Moreover, Lemma~\ref{bbblem} implies
 \begin{equation}\label{wirdgut}
 A_{B_1 B_2}-\lambda \subset (\widetilde A_{B_1'B_2'})^*-\lambda 
 \end{equation}
  and it is clear that $(\widetilde A_{B_1'B_2'})^*-\lambda$ is closed. Furthermore, is $(\widetilde A_{B_1'B_2'})^*-\lambda$ injective, since we have 
  $$\ker \bigl((\widetilde A_{B_1'B_2'})^*-\lambda \bigr)=\ran\bigl(\widetilde A_{B_1'B_2'}-\myoverline\lambda\bigr)^\bot=\{0\}.$$
  Since $A_{B_1 B_2}-\lambda$ is bijective it follows that the operators in \eqref{wirdgut} coincide and hence we conclude \eqref{nabitte}. In particular, 
  $\lambda\in\rho(A_{B_1 B_2})$.
\end{proof}

Now we provide more direct and explicit criteria on the Weyl function and the parameter $B_1B_2$ such that $A_{B_1 B_2}$ becomes 
a closed operator with a nonempty resolvent set. We do not formulate a variant of Theorem~\ref{bsextendedsatz} for the operators $\widetilde A_{\widetilde B_1 \widetilde B_2}$,
and we also leave it to the reader to formulate the corresponding versions of Corollary~\ref{cor1} and Corollary~\ref{bsextendedcorchen} below.

\begin{theorem}\label{bsextendedsatz}
Assume that the triple
 $\{\cG,(\Gamma_0,\Gamma_1),(\widetilde\Gamma_0,\widetilde\Gamma_1)\}$ satisfies {\rm (G)}, {\rm (D)}, {\rm (M)}, and that the resolvent set of 
 $A_0$  is nonempty. 
Let $\gamma,\widetilde\gamma$ and $M$ be the associated $\gamma$-fields and Weyl function, respectively.
Assume that $B_1$ and $B_2$ are closable operators in $\cG$ and that for some $\lambda_0\in\rho(A_0)$ the following conditions hold: 
\begin{enumerate}
\item[{\rm (i)}] $1\in\rho(B_2\myoverline{M(\lambda_0)B_1})$;
\item[{\rm (ii)}] $\ran(B_2\myoverline{M(\lambda_0)B_1})\subset\ran\Gamma_0\cap\dom B_1$;
\item[{\rm (iii)}] $\ran(B_1\upharpoonright\ran\Gamma_0)\subset\ran\Gamma_0$;
\item[{\rm (iv)}] $\ran(B_2\upharpoonright\ran\Gamma_1)\subset\ran\Gamma_0$; 
\item[{\rm (v)}] $\ran(\Gamma_1\upharpoonright\ker\Gamma_0)\subset\dom B_1 B_2$.
\end{enumerate}
Then  $A_{B_1 B_2}$  in \eqref{abs1212}
is a closed operator with a nonempty resolvent set and for all $\lambda\in\rho(A_0)\cap\rho(A_{B_1B_2})$ the Krein-type resolvent formula
\begin{equation}\label{Eq_Krein_formula}
(A_{B_1B_2}-\lambda)^{-1}=(A_0-\lambda)^{-1}+\gamma(\lambda)B_1\bigl(I-B_2M(\lambda)B_1\bigr)^{-1}B_2\widetilde\gamma(\myoverline{\lambda})^*
\end{equation}
is valid.
\end{theorem}

\begin{proof}
We verify the inclusion
\begin{equation}\label{inc11}
\ran\bigl(B_2\widetilde\gamma(\myoverline\lambda_0)^*\bigr)\subset\ran\bigl(I-B_2M(\lambda_0)B_1\bigr).
\end{equation}
In fact, consider some $\psi\in\ran(B_2\widetilde\gamma(\myoverline\lambda_0)^*)$. Then 
$$\psi=B_2\widetilde\gamma(\myoverline\lambda_0)^*f=B_2\Gamma_1(A_0-\lambda_0)^{-1}f$$ 
for some $f\in\sH$ by Proposition~\ref{gamprop}~(iii)
and from $\dom A_0=\ker\Gamma_0$ and conditions (iv)--(v) we  obtain $\psi\in\ran\Gamma_0\cap\dom B_1$. By condition (i) 
\begin{equation}\label{yyy}
\varphi:=\bigl(I-B_2\myoverline{M(\lambda_0)B_1}\bigr)^{-1}\psi
\end{equation}
is well defined and $\varphi-\psi=B_2\myoverline{M(\lambda_0)B_1}\varphi\in\ran\Gamma_0\cap\dom B_1$ by (ii). Hence also $\varphi\in\ran\Gamma_0\cap\dom B_1$
and (iii) implies $B_1\varphi\in\ran\Gamma_0=\dom M(\lambda_0)$. Therefore, 
$B_2\myoverline{M(\lambda_0)B_1}\varphi=B_2M(\lambda_0)B_1\varphi$ and together with \eqref{yyy} we conclude 
\begin{equation*}
\bigl(I-B_2M(\lambda_0)B_1\bigr)\varphi=\psi,
\end{equation*}
which shows \eqref{inc11}.

It is clear from condition (i) and Theorem~\ref{bsthm}~(i) that $\lambda_0\not\in\sigma_p(A_{B_1B_2})$ and by the above observation we can apply Theorem~\ref{kreinthm}~(i)
for $\lambda_0\in\rho(A_0)$. More precisely, for any $f\in\sH$ we have $\widetilde\gamma(\myoverline\lambda_0)^*f\in\dom B_2$ by condition (v) and 
$B_2\widetilde\gamma(\myoverline\lambda_0)^*f\in \ran(I-B_2M(\lambda_0)B_1)$ was shown in \eqref{inc11}. Hence \eqref{asskrein} is valid for all $f\in\sH$ and 
\begin{equation}\label{Eq_Krein_formulalambda0}
(A_{B_1B_2}-\lambda_0)^{-1}f=(A_0-\lambda_0)^{-1}f+\gamma(\lambda_0)B_1\bigl(I-B_2M(\lambda_0)B_1\bigr)^{-1}B_2\widetilde\gamma(\myoverline\lambda_0)^*f
\end{equation}
holds. Moreover, as $B_2$ is closable and $\widetilde\gamma(\myoverline\lambda_0)^*$ is everywhere defined and bounded (see Proposition~\ref{gamprop}~(iii)) it follows that $B_2\widetilde\gamma(\myoverline\lambda_0)^*$ is closable and hence closed and everywhere defined, and thus bounded. Similarly, condition (i) and the assumption that $B_1$ is closable 
imply that $B_1(I-B_2\myoverline{M(\lambda_0)B_1})^{-1}$ is everywhere defined and bounded and hence the restriction $B_1(I-B_2 M(\lambda_0)B_1)^{-1}$ is also bounded.
Furthermore, $\gamma(\lambda_0)$ is a bounded operator by Proposition~\ref{gamprop}~(i). Summing up we have shown that
\begin{equation*}
 \gamma(\lambda_0)B_1\bigl(I-B_2M(\lambda_0)B_1\bigr)^{-1}B_2\widetilde\gamma(\myoverline\lambda_0)^*
\end{equation*}
is a bounded and everywhere defined operator. The same is true for $(A_0-\lambda_0)^{-1}$ and from \eqref{Eq_Krein_formulalambda0} we conclude that 
$(A_{B_1B_2}-\lambda_0)^{-1}$ is a bounded and everywhere defined operator, and hence closed. This implies that $A_{B_1B_2}$ is closed and $\lambda_0\in\rho(A_{B_1B_2})$.

Now consider $\lambda\in\rho(A_{B_1B_2})\cap\rho(A_0)$. As above we have for any $f\in\sH$ that $\widetilde\gamma(\myoverline\lambda)^*f\in\dom B_2$ by condition (v). We claim that
\begin{equation}\label{jaja}
 B_2\widetilde\gamma(\myoverline\lambda)^*f\in\ran\bigl(I-B_2M(\lambda)B_1\bigr).
\end{equation}
For this we consider $k=(A_{B_1B_2}-\lambda)^{-1}f$ and $h=(A_0-\lambda)^{-1}f$. Note that
$B_1B_2\Gamma_1 k=\Gamma_0 k$ and, in particular, $\Gamma_1 k\in\dom B_2$. Moreover, $\Gamma_0 h=0$ and from
$(T-\lambda)(k-h)=0$ we conclude $M(\lambda)\Gamma_0(k-h)=\Gamma_1(k-h)$. Therefore,
\begin{equation}\label{mmm}
 M(\lambda)B_1 B_2 \Gamma_1 k=M(\lambda)\Gamma_0 k=M(\lambda)\Gamma_0 (k-h)=\Gamma_1 (k-h).
\end{equation}
As $\Gamma_1 k\in\dom B_2$ and $\Gamma_1 h=\Gamma_1 (A_0-\lambda)^{-1}f=\widetilde\gamma(\myoverline\lambda)^*f\in\dom B_2$ by (v) 
we see that $M(\lambda)B_1 B_2 \Gamma_1 k\in\dom B_2$ and hence the element
$$
B_2M(\lambda)B_1 B_2 \Gamma_1 k
$$
is well defined. Now we use \eqref{mmm} and $\Gamma_1 h=\widetilde\gamma(\myoverline\lambda)^*f$ and compute
\begin{equation*}
 \bigl(I-B_2M(\lambda)B_1\bigr)B_2 \Gamma_1 k=B_2 \Gamma_1 k- B_2\Gamma_1 (k-h)= B_2\Gamma_1 h=B_2\widetilde\gamma(\myoverline\lambda)^*f,
\end{equation*}
which shows \eqref{jaja}. Therefore, both conditions in \eqref{asskrein} are satisfied for all $\lambda\in\rho(A_{B_1B_2})\cap\rho(A_0)$ and $f\in\sH$,
and hence the Krein-type resolvent formula
\eqref{Eq_Krein_formula} follows from Theorem~\ref{kreinthm}.
\end{proof}

In the special case $B_1=I$ and $B_2=B$ one obtains the following statement.

\begin{corollary}\label{cor1}
Assume that the triple
 $\{\cG,(\Gamma_0,\Gamma_1),(\widetilde\Gamma_0,\widetilde\Gamma_1)\}$ satisfies {\rm (G)}, {\rm (D)}, {\rm (M)}, and that the resolvent set of 
 $A_0$ or, equivalently, the resolvent set of $\widetilde A_0$ is nonempty. 
Let $\gamma,\widetilde\gamma$ and $M,\widetilde M$ be the associated $\gamma$-fields and Weyl functions, respectively.
Assume that $B$ is a closable operator in $\cG$ and that for some $\lambda_0\in\rho(A_0)$ the following conditions hold: 
\begin{enumerate}
\item[{\rm (i)}] $1\in\rho(B\myoverline{M(\lambda_0)})$;
\item[{\rm (ii)}] $\ran(B\myoverline{M(\lambda_0)})\subset\ran\Gamma_0$;
\item[{\rm (iii)}] $\ran(B\upharpoonright\ran\Gamma_1)\subset\ran\Gamma_0$; 
\item[{\rm (iv)}] $\ran(\Gamma_1\upharpoonright\ker\Gamma_0)\subset\dom B$.
\end{enumerate}
Then $A_{B}$ in \eqref{abs}
is a closed operator with a nonempty resolvent set and for all $\lambda\in\rho(A_0)\cap\rho(A_{B})$ the Krein-type resolvent formula
\begin{equation}\label{malwieder}
(A_{B}-\lambda)^{-1}=(A_0-\lambda)^{-1}+\gamma(\lambda)\bigl(I-BM(\lambda)\bigr)^{-1}B\widetilde\gamma(\myoverline{\lambda})^*
\end{equation}
is valid.
\end{corollary}

In the next corollary the special case $\ran\Gamma_0=\ran\widetilde\Gamma_0=\cG$ is considered. Recall from Lemma~\ref{ddlem}, 
that in this situation $\{\cG,(\Gamma_0,\Gamma_1),(\widetilde\Gamma_0,\widetilde\Gamma_1)\}$ is a quasi boundary triple (or even 
generalized boundary triple) if also {\rm (G)} and {\rm (M)} are required.

\begin{corollary}\label{bsextendedcorchen}
Assume that the triple
 $\{\cG,(\Gamma_0,\Gamma_1),(\widetilde\Gamma_0,\widetilde\Gamma_1)\}$ satisfies {\rm (G)}, 
 $\ran\Gamma_0=\ran\widetilde\Gamma_0=\cG$, {\rm (M)}, and that the resolvent set of 
 $A_0$ or, equivalently, the resolvent set of $\widetilde A_0$ is nonempty. 
Assume that $B_1$ and $B_2$ are closable operators in $\cG$ and that for some $\lambda_0\in\rho(A_0)$ the following conditions hold: 
\begin{enumerate}
\item[{\rm (i)}] $1\in\rho(B_2\myoverline{M(\lambda_0)B_1})$;
\item[{\rm (ii)}] $\ran(B_2\myoverline{M(\lambda_0)B_1})\subset\dom B_1$;
\item[{\rm (iii)}] $\ran(\Gamma_1\upharpoonright\ker\Gamma_0)\subset\dom B_1 B_2$.
\end{enumerate}
Then  $A_{B_1 B_2}$  in \eqref{abs1212}
is a closed operator with a nonempty resolvent set and for all $\lambda\in\rho(A_0)\cap\rho(A_{B_1B_2})$ the Krein-type resolvent formula \eqref{Eq_Krein_formula}
is valid. In the special case $B_1=I$ and $B_2=B$ the conditions (i)--(iii) reduce to  
\begin{enumerate}
\item[{\rm (i)}] $1\in\rho(B M(\lambda_0))$;
\item[{\rm (ii)}] $\ran(\Gamma_1\upharpoonright\ker\Gamma_0)\subset\dom B$;
\end{enumerate}
and $A_{B}$  in \eqref{abs}
is a closed operator with a nonempty resolvent set and for all $\lambda\in\rho(A_0)\cap\rho(A_{B})$ the Krein-type resolvent formula \eqref{malwieder} is valid.
\end{corollary}

We briefly return to the setting in Example~\ref{exi1} at the end of Section~\ref{sec2}.

\begin{example}\label{exi2}
The $\gamma$-fields and Weyl functions associated with the quasi boundary triple in Example~\ref{exi1} are given by
\begin{equation*}
\begin{split}
 \gamma(\lambda)\varphi&=f_\lambda(\varphi),\qquad \varphi\in L^2(\partial\Omega,\dC^{m\times m}),\,\,\,\,\lambda\in\rho(A_0),\\
 \widetilde\gamma(\mu)\psi&=g_\mu(\psi),\qquad \psi\in L^2(\partial\Omega,\dC^{m\times m}),\,\,\,\,\mu\in\rho(\widetilde A_0),
\end{split}
\end{equation*}
and

\begin{equation*}
\begin{split}
 M(\lambda)\varphi&=-\iota_-\tau_N f_\lambda(\varphi),\qquad \varphi\in L^2(\partial\Omega,\dC^{m\times m}),\,\,\,\,\lambda\in\rho(A_0),\\
 \widetilde M(\mu)\psi&=-\iota_-\widetilde\tau_N g_\mu(\psi),\qquad \psi\in L^2(\partial\Omega,\dC^{m\times m}),\,\,\,\,\mu\in\rho(\widetilde A_0),
\end{split}
\end{equation*}
where $f_\lambda(\varphi), g_\mu(\psi)\in H^1(\Omega,\dC^{m\times m})$ are the unique solutions of the boundary value problems
\begin{equation*}
\begin{split}
 \cP f_\lambda(\varphi)&=\lambda f_\lambda(\varphi),\quad \iota_+\tau_D f_\lambda(\varphi)=\varphi,\\
 \widetilde\cP g_\mu(\psi)&=\mu g_\mu(\psi),\quad \iota_+\tau_D g_\mu(\psi)=\psi,
\end{split}
 \end{equation*}
respectively. Note that 
\begin{equation*}
\begin{split}
-\iota_-^{-1} M(\lambda)\iota_+&: H^{1/2}(\partial\Omega,\dC^{m\times m})\rightarrow H^{-1/2}(\partial\Omega,\dC^{m\times m}),\,\,\,\,\lambda\in\rho(A_0), \\
-\iota_-^{-1} \widetilde M(\mu)\iota_+&: H^{1/2}(\partial\Omega,\dC^{m\times m})\rightarrow H^{-1/2}(\partial\Omega,\dC^{m\times m}),\,\,\,\,\mu\in\rho(\widetilde A_0),
\end{split}
\end{equation*}
are the Dirichlet-to-Neumann maps corresponding to the differential expressions $\cP-\lambda$ and $\widetilde\cP-\mu$, respectively.

It follows from Corollary~\ref{bsextendedcorchen} that if, e.g., $B$ is an everywhere defined bounded operator in $L^2(\partial\Omega,\dC^{m\times m})$ such that 
$1\in\rho(B M(\lambda_0))$ for some $\lambda_0\in\rho(A_0)$, then 
\begin{equation*}
 \begin{split}
  A_B f&= \cP f,\\ 
  \dom A_B&=\bigl\{f\in H^1(\Omega,\dC^{m\times m}): \cP f \in L^2(\Omega,\dC^{m\times m}),\, \iota_+\tau_D f+ B \iota_-\tau_N f = 0\bigr\},
  \end{split}
 \end{equation*}
is a closed operator in $L^2(\Omega,\dC^{m\times m})$ with a nonempty resolvent set; it is clear that for $\lambda\in\rho(A_B)$ and $h\in L^2(\Omega,\dC^{m\times m})$ the unique $H^1(\Omega,\dC^{m\times m})$-solution of the boundary value problem
\begin{equation*}
 (\cP-\lambda)f = h,\qquad \iota_+\tau_D f+ B \iota_-\tau_N f = 0,
\end{equation*}
is given by $f=(A_B-\lambda)^{-1}h$. Furthermore, if $\lambda\in\rho(A_0)\cap\rho(A_{B})$, then the solution can be expressed via the Krein-type resolvent formula \eqref{malwieder}.
We leave it to the reader to formulate a variant of this observation for Robin-type realizations of the adjoint differential expression $\widetilde \cP$.
\end{example}

\begin{appendix}
\section{The special case $S=\widetilde S$}\label{app}

We provide a summary of our results in the special situation that $S$ is a densely defined closed symmetric operator, that is, $\{S,S\}$ is an adjoint pair.
In this case one can choose $T=\widetilde T$ and $\Gamma_0=\widetilde\Gamma_0$, $\Gamma_1=\widetilde\Gamma_1$.
The results below are known from \cite{BL07,BL12,BLLR18,BS19} for the special case that 
$\{\cG,\Gamma_0,\Gamma_1\}$ is a quasi boundary triple, but mostly remain valid under the weaker assumptions 
{\rm ($\sG$)}, {\rm ($\sD$)}, or {\rm ($\sM$)}; cf. Definition~\ref{qbt}, which reduces to the following:

\begin{definition}\label{qbtsym}
Let $S$ be a densely defined closed symmetric operator in $\sH$ and assume that
$T$ is a core of $S^*$.
We shall consider {\em triples} of the form $\{\cG,\Gamma_0,\Gamma_1\}$ for $S$, where
$\cG$ is a Hilbert space and
\begin{equation*}
 \Gamma_0,\Gamma_1:\dom T\rightarrow\cG
\end{equation*}
are linear mappings such that
\begin{itemize}
 \item [\rm ($\sG$)] the abstract Green's identity
\begin{equation*}
 (Tf,g)_\cH-(f, Tg)_\cH=(\Gamma_1 f, \Gamma_0 g)_\cG-(\Gamma_0 f,\Gamma_1 g)_\cG
\end{equation*}
holds for all $f,g\in\dom T$,
\item [{\rm ($\sD$)}] the range of $\Gamma_0:\dom T\rightarrow\cG$ is dense,
\item [{\rm ($\sM$)}] the operator $A_0:=T\upharpoonright\ker\Gamma_0$ is self-adjoint in $\sH$.
\end{itemize}
If $\{\cG,\Gamma_0,\Gamma_1\}$ is such that {\rm ($\sG$)},
\begin{itemize}
\item [{\rm ($\sD\sD$)}] the range of $(\Gamma_0,\Gamma_1)^\top:\dom T\rightarrow\cG\times \cG$ is dense,
\end{itemize}
and {\rm ($\sM$)} hold, then $\{\cG,\Gamma_0,\Gamma_1\}$ is said to be a {\em quasi boundary triple} for $S$.
\end{definition}

Note that in the present situation the natural counterpart of the maximality condition (M) in Definition~\ref{qbt}
is the requirement in ($\sM$) that $A_0=T\upharpoonright\ker\Gamma_0$ is self-adjoint in $\sH$. 
We also mention that $A_0$ is symmetric whenever ($\sG$) holds; cf. Remark~\ref{remmilein}.

\begin{lemma}\label{gbtlem}
Assume that the triple $\{\cG,\Gamma_0,\Gamma_1\}$ satisfies {\rm ($\sG$)} and {\rm ($\sM$)}.
If $\ran\Gamma_0=\cG$, then $\ran(\Gamma_0,\Gamma_1)^\top$ is dense in $\cG\times \cG$ and $\{\cG,\Gamma_0,\Gamma_1\}$  is a quasi boundary triple for $S$.
\end{lemma}

The statement in Lemma~\ref{gbtlem} is known from
\cite[Lemma 6.1]{DM95};
in this situation the quasi boundary triple $\{\cG,\Gamma_0,\Gamma_1\}$ is even a so-called 
generalized boundary triple in the sense of \cite[Section~6]{DM95},
see also \cite{DHMS06,DHMS12}.

\begin{lemma}
Assume that the triple $\{\cG,\Gamma_0,\Gamma_1\}$ satisfies {\rm ($\sG$)}, {\rm ($\sD$)}, and {\rm ($\sM$)}.
 Then
 \begin{equation*}
  \dom S=\ker\Gamma_0\cap\ker\Gamma_1.
 \end{equation*}
\end{lemma}

\begin{lemma}
 Assume that the triple
 $\{\cG,\Gamma_0,\Gamma_1\}$ satisfies {\rm ($\sG$)} and {\rm ($\sD\sD$)}. Then the mapping
 \begin{equation*}
  \begin{pmatrix} \Gamma_0\\ \Gamma_1\end{pmatrix}:\dom T\rightarrow\cG\times \cG
 \end{equation*}
is closable with respect to the graph norm of $T$. In particular, the individual mappings $\Gamma_0,\Gamma_1:\dom T\rightarrow\cG$ 
are closable.
\end{lemma}

Proposition~\ref{ordiprop} is known for quasi boundary triples from \cite[Theorem 2.3]{BL07}, where also a variant of Theorem~\ref{ratethm} is contained in the symmetric setting. 
Here we only state the symmetric version of Proposition~\ref{ordiprop}, which in fact leads to an ordinary boundary triple; cf. \cite{BHS20,DM91}

\begin{proposition}
 Let 
 $\{\cG,\Gamma_0,\Gamma_1\}$ be a quasi boundary triple for $S$. Then
 the following are equivalent:
 \begin{itemize}
  \item [{\rm (i)}] $T=S^*$, 
  \item [{\rm (ii)}] $\ran(\Gamma_0,\Gamma_1)^\top=\cG\times\cG$.
 \end{itemize}
\end{proposition} 

The $\gamma$-field and Weyl function corresponding to a triple $\{\cG,\Gamma_0,\Gamma_1\}$ are introduced in the same way as in Section~\ref{gamsec} using the 
decomposition
\begin{equation*}
 \dom T=\dom A_0\dot + \ker(T-\lambda)=\ker\Gamma_0\dot+\ker(T-\lambda)
\end{equation*}
for $\lambda\in\rho(A_0)$ as follows:
\begin{equation*}
 \gamma(\lambda)=\bigl(\Gamma_0\upharpoonright\ker(T-\lambda)\bigr)^{-1}\quad\text{and}\quad 
 M(\lambda)=\Gamma_1\bigl(\Gamma_0\upharpoonright\ker(T-\lambda)\bigr)^{-1}.
\end{equation*}

In the symmetric situation Proposition~\ref{gamprop} and Proposition~\ref{mprop} reduce to the following statements.

\begin{proposition}
Assume that the triple
 $\{\cG,\Gamma_0,\Gamma_1\}$ satisfies {\rm ($\sG$)}, {\rm ($\sD$)}, {\rm ($\sM$)}, and let 
 $\gamma$  be the corresponding $\gamma$-field.
Then the following assertions hold for all $\lambda\in\rho(A_0)$.
\begin{itemize}
 \item [{\rm (i)}] $\gamma(\lambda)$ is a bounded operator from $\cG$ into $\sH$ with dense domain
 $\dom\gamma(\lambda)=\ran\Gamma_0$ and $\ran\gamma(\lambda)=\ker(T-\lambda)$;
 \item [{\rm (ii)}] for $\varphi\in\ran\Gamma_0$ the function $\lambda\mapsto\gamma(\lambda)\varphi$ is holomorphic 
 on $\rho(A_0)$ and 
 \begin{equation*}
   \gamma(\lambda)=\bigl(I+(\lambda-\nu)(A_0-\lambda)^{-1}\bigr)\gamma(\nu),\qquad \lambda,\nu\in\rho(A_0),
\end{equation*}
holds;
 \item [{\rm (iii)}] $\gamma(\lambda)^*$ is an everywhere defined bounded operator from $\sH$ to $\cG$
 and for all $f\in\sH$ one has 
 \begin{equation*}
  \gamma(\lambda)^*f=\Gamma_1(A_0-\myoverline\lambda)^{-1}f,
 \end{equation*}
 in particular, $\ran\gamma(\lambda)^*\subset\ran\Gamma_1$.
\end{itemize}
\end{proposition}

\begin{proposition}
Assume that the triple
 $\{\cG,\Gamma_0,\Gamma_1\}$ satisfies {\rm ($\sG$)}, {\rm ($\sD$)}, {\rm ($\sM$)}, and let $\gamma$ and $M$ be the corresponding $\gamma$-field and Weyl function, respectively.
Then the following assertions hold for all $\lambda,\mu\in\rho(A_0)$.
\begin{itemize}
 \item [{\rm (i)}] $M(\lambda)$ is an operator in $\cG$ with dense domain
 $\dom M(\lambda)=\ran\Gamma_0$ and $\ran M(\lambda)\subset\ran \Gamma_1$;
\item [{\rm (ii)}] $M(\lambda)\subset M(\myoverline\lambda)^*$ and one has the identities
\begin{equation*}
 M(\lambda)- M(\mu)^*=(\lambda-\myoverline\mu)\gamma(\mu)^*\gamma(\lambda)
\end{equation*}
\item [{\rm (iii)}]
The function $\lambda\mapsto M(\lambda)$ is holomorphic in the sense that it can be
written as the sum of the possibly unbounded closed operator $M(\lambda_0)^*$,
where $\lambda_0\in\rho(A_0)$ is fixed, and a bounded holomorphic operator
function:
\begin{equation*}
 M(\lambda)= M(\lambda_0)^*+\gamma(\lambda_0)^*(\lambda-\myoverline\lambda_0)\bigl(I+(\lambda-\lambda_0)(A_0-\lambda)^{-1}\bigr)\gamma(\lambda_0).
\end{equation*}
\end{itemize}
\end{proposition}

Next we state the abstract Birman-Schwinger principle in Theorem~\ref{bsthm} in a symmetrized form for the operator 
\begin{equation}\label{abs12122}
  A_{B_1B_2} f= Tf,\qquad \dom A_{B_1B_2}=\bigl\{f\in\dom T: B_1B_2\Gamma_1 f = \Gamma_0 f\bigr\}
 \end{equation}
where $ B_1B_2$ is a (product of) linear operators in $\cG$. The special case $B_2=B$, $B_1=I$ is left to the reader.

\begin{theorem}
Consider a triple
 $\{\cG,\Gamma_0,\Gamma_1\}$ as in Definition~\ref{qbtsym}, assume that 
 the resolvent set of  $A_0$ is nonempty, and let  $M$  be the associated Weyl function.
 Furthermore, let $A_{B_1 B_2}$ be the operator in \eqref{abs12122} and let $\lambda\in\rho(A_0)$. Then
 $\lambda\in\sigma_p(A_{B_1 B_2})$ if and only if $\ker(I-B_2M(\lambda)B_1)\not=\{0\}$, and in this case 
 \begin{equation*}
  \ker(A_{B_1 B_2}-\lambda)=\bigl\{f_\lambda\in\ker(T-\lambda):\Gamma_0 f_\lambda= B_1\varphi, 
  \varphi\in \ker(I-B_2M(\lambda)B_1)\bigr\}.
 \end{equation*}

\end{theorem}

Theorem~\ref{kreinthm} has the following form in the symmetric case.

\begin{theorem}
Assume that the triple
 $\{\cG,\Gamma_0,\Gamma_1\}$ satisfies {\rm ($\sG$)}, {\rm ($\sD$)}, {\rm ($\sM$)}, and let
$\gamma$ and $M$ be the associated $\gamma$-field and Weyl function, respectively.
Furthermore, let $A_{B_1 B_2}$ be the operator in \eqref{abs12122} and let $\lambda\in\rho(A_0)$.  If $\lambda\not\in\sigma_p(A_{B_1B_2})$ and $f\in\sH$ is such that
\begin{equation}\label{asskrein33}
\gamma(\myoverline\lambda)^*f\in\dom B_2\quad\text{and}\quad B_2\gamma(\myoverline\lambda)^*f\in\ran(I-B_2M(\lambda)B_1),
\end{equation}
then $f\in\ran(A_{B_1B_2}-\lambda)$ and the Krein-type formula 
\begin{equation*}
 (A_{B_1B_2}-\lambda)^{-1}f=(A_0-\lambda)^{-1}f+\gamma(\lambda)B_1\bigl(I-B_2M(\lambda)B_1\bigr)^{-1}B_2\gamma(\myoverline\lambda)^*f
\end{equation*}
holds.
In particular, if \eqref{asskrein33} holds for all $f\in\sH$, that is, $
\ran \gamma(\myoverline\lambda)^*\subset \dom B_2$ and $\ran B_2\gamma(\myoverline\lambda)^*\subset \ran(I-B_2M(\lambda)B_1)$,
then $A_{B_1B_2}-\lambda$ is a bijective operator in $\sH$.
\end{theorem}

In the special case that $B_1B_2$ is a symmetric operator and $\{\cG,\Gamma_0,\Gamma_1\}$ is a quasi boundary triple the next result is known from \cite[Theorem 2.2 and Remark 2.5]{BS19}. 

\begin{theorem}
Assume that the triple
 $\{\cG,\Gamma_0,\Gamma_1\}$ satisfies {\rm ($\sG$)}, {\rm ($\sD$)}, {\rm ($\sM$)}, and let
$\gamma$ and $M$ be the associated $\gamma$-field and Weyl function, respectively.
Assume that $B_1$ and $B_2$ are closable operators in $\cG$ and that for some $\lambda_0\in\rho(A_0)$ the following conditions hold: 
\begin{enumerate}
\item[{\rm (i)}] $1\in\rho(B_2\myoverline{M(\lambda_0)B_1})$;
\item[{\rm (ii)}] $\ran(B_2\myoverline{M(\lambda_0)B_1})\subset\ran\Gamma_0\cap\dom B_1$;
\item[{\rm (iii)}] $\ran(B_1\upharpoonright\ran\Gamma_0)\subset\ran\Gamma_0$;
\item[{\rm (iv)}] $\ran(B_2\upharpoonright\ran\Gamma_1)\subset\ran\Gamma_0$; 
\item[{\rm (v)}] $\ran(\Gamma_1\upharpoonright\ker\Gamma_0)\subset\dom B_1 B_2$.
\end{enumerate}
Then  $A_{B_1 B_2}$  in \eqref{abs12122}
is a closed operator with a nonempty resolvent set and for all $\lambda\in\rho(A_0)\cap\rho(A_{B_1B_2})$ the Krein-type resolvent formula
\begin{equation*}
(A_{B_1B_2}-\lambda)^{-1}=(A_0-\lambda)^{-1}+\gamma(\lambda)B_1\bigl(I-B_2M(\lambda)B_1\bigr)^{-1}B_2\gamma(\myoverline{\lambda})^*
\end{equation*}
is valid. If, in addition, the parameter $B_1B_2$ is a symmetric operator in $\cG$ and the conditions (i)-(v) hold for
some $\lambda_0\in\rho(A_0)\cap\dR$ or some $\lambda_\pm\in\dC^\pm$, then $A_{B_1 B_2}$ is self-adjoint in $\sH$.
\end{theorem}
\end{appendix}


\begin{thebibliography}{33}


\bibitem{AGW14} H.~Abels, G.~Grubb and I.\,G.~Wood,
Extension theory and Kre\u{\i}n-type resolvent formulas for nonsmooth
boundary value problems,
J.\ Funct.\ Anal. 266 (2014), 4037--4100.

\bibitem{AB09} D.\ Alpay and J.\ Behrndt, Generalized $Q$-functions and Dirichlet-to-Neumann maps for elliptic differential operators, J. Funct. Anal. 257 (2009), 1666--1694.

\bibitem{ACE23} W. Arendt, I. Chalender, and R. Eymard,
Extensions of dissipative and symmetric operators,
Semigroup Forum 106 (2023), 339--367.

\bibitem{AE12} W. Arendt and A.\,F.\,M.~ter Elst, Sectorial forms and degenerate differential operators. J. Operator Theory 67 (2012), 33--72.

\bibitem{AEKS14} W. Arendt, A.\,F.\,M.~ter Elst, J.\,B.~Kennedy, and M. Sauter, The Dirichlet-to-Neumann operator via hidden compactness, J. Funct. Anal. 266 (2014), 1757--1786. 

\bibitem{A00} Y. Arlinskii,
Abstract boundary conditions for maximal sectorial extensions of sectorial operators,
Math. Nachr. 209 (2000), 5--36.

\bibitem{A12} Y. Arlinskii,
Boundary triplets and maximal accretive extensions of sectorial operators,
In {\em Operator methods for boundary value problems},
London Math. Soc. Lecture Note Ser. 404, 35--72 Cambridge Univ. Press, Cambridge, (2012).

\bibitem{AP17} Y. Arlinskii and A. Popov, On $m$-sectorial extensions of sectorial operators,
J. Math. Phys. Anal. Geom. 13 (2017), 205--241.

\bibitem{BEHL18}
J. Behrndt, P. Exner, M. Holzmann, and V. Lotoreichik,
On the spectral properties of Dirac operators with electrostatic $\delta$-shell interactions,
J. Math. Pures Appl. 111 (2018), 47--78.

\bibitem{BGHN17} J. Behrndt, F. Gesztesy, H. Holden, and R. Nichols, 
On the index of meromorphic operator-valued functions and some applications, in:
Functional Analysis and Operator Theory for Quantum Physics - The Pavel Exner Anniversary Volume, EMS (2017), 95--127.


\bibitem{BHS20}
J.~Behrndt, S.~Hassi, and H.S.V.~de~Snoo,
 \textit{Boundary Value Problems, Weyl Functions, and Differential Operators},
 Monographs in Mathematics, Vol. 108, Birkh\"auser, 2020.

\bibitem{BL07} J. Behrndt and M. Langer, Boundary value problems for elliptic partial differential operators on bounded domains, J. Funct. Anal. 2443 (2007), 536--565.

\bibitem{BL12} J.~Behrndt and M.~Langer, Elliptic operators, Dirichlet-to-Neumann maps and quasi boundary triples, 
In {\em Operator methods for boundary value problems},
London Math. Soc. Lecture Note Ser. 404, 121--160, Cambridge Univ. Press, Cambridge, (2012).


\bibitem{BLL13} J.~Behrndt, M.~Langer, and V.~Lotoreichik, Schr\"odinger operators with $\delta$ and $\delta'$-potentials supported on hypersurfaces, Ann.\ Henri Poincar\'e 14 (2013), 385--423.

\bibitem{BLLR18} J.~Behrndt, M.~Langer, V. Lotoreichik, and J. Rohleder,
Spectral enclosures for non-self-adjoint extensions of symmetric operators,
J. Funct. Anal. 275 (2018), 1808--1888.

\bibitem{BMN17}
J. Behrndt, M.M. Malamud, and H. Neidhardt,
Scattering matrices and Dirichlet-to-Neumann maps.
J. Funct. Anal. 273 (2017), 1970--2025.

\bibitem{BM14} J. Behrndt and T. Micheler, Elliptic differential operators on Lipschitz domains and abstract boundary value problems, J. Funct. Anal. 267 (2014), 3657--3709.

\bibitem{BS19} J.~Behrndt and P.~Schlosser,
Quasi boundary triples, self-adjoint extensions, and Robin Laplacians on the half-space,
Oper. Theory Adv. Appl. 275 (2019), 49--66.



%
%
%
%
%
%
\bibitem{B62} M.S.~Birman, Perturbations of the continuous spectrum of a singular elliptic operator by varying the boundary and the boundary conditions, Vestnik Leningrad.\ Univ. 17 (1962), 22--55 (in Russian); translated in: Amer.\ Math.\ Soc.\ Transl. 225 (2008), 19--53.
%
%


\bibitem{BGW09} B.M.~Brown, G.~Grubb, and I.G.~Wood, $M$-functions for closed extensions of adjoint pairs of operators with applications to elliptic boundary problems, 
Math.\ Nachr. 282 (2009), 314--347.

\bibitem{BHMNW09} B.M.~Brown, J.~Hinchcliffe, M.~Marletta, S.~Naboko, and I.~Wood, The abstract Titchmarsh--Weyl $M$-function for adjoint operator pairs and its relation to the spectrum, Integral Equations Operator Theory 63 (2009), 297--320.

\bibitem{BMNW08} B.M.~Brown, M.~Marletta, S.~Naboko, and I.G.~Wood, Boundary triplets and $M$-functions for non-selfadjoint operators, with applications to elliptic PDEs and block operator matrices,  J. Lond. Math. Soc. (2)  77  (2008), 700--718.

\bibitem{B76} V.M.~Bruk, A certain class of boundary value problems with a spectral parameter in the boundary condition, Mat.\ Sb.\ (N.S.) 100 (142) (1976), 210--216 (in Russian); translated in: Math.\ USSR-Sb. 29 (1976), 186--192.

\bibitem{BGP08} J.~Br\"{u}ning, V.~Geyler, and K.~Pankrashkin, Spectra of selfadjoint extensions and applications to solvable Schr\"{o}dinger operators, Rev.\ Math.\ Phys. 20 (2008), 1--70.
%

\bibitem{DHM20}
V.A. Derkach, S. Hassi, and M.M. Malamud,
Generalized boundary triples, I. Some classes of isometric and unitary boundary pairs and realization problems for subclasses of Nevanlinna functions,
Math. Nachr. 293 (2020), 1278--1327.


\bibitem{DHM22}
V.A. Derkach, S. Hassi, and M.M. Malamud,
Generalized boundary triples, II. Some applications of generalized boundary triples and form domain invariant Nevanlinna functions,
Math. Nachr. 295 (2022), 1113--1162.

\bibitem{DHMS00}
V.A. Derkach, S. Hassi, M.M. Malamud, and H. de~Snoo, 
Generalized resolvents of symmetric operators and admissibility,
Methods Funct. Anal. Topology 6 (2000) 24--55.

\bibitem{DHMS06}
V.A. Derkach, S. Hassi, M.M. Malamud, and H. de~Snoo, 
Boundary relations and their Weyl families.
Trans. Amer. Math. Soc. 358 (2006), 5351--5400.

\bibitem{DHMS09}
V.A. Derkach, S. Hassi, M.M. Malamud, and H. de~Snoo,
Boundary relations and generalized resolvents of symmetric operators.
Russian J. Math. Phys. 16 (2009) 17--60.

\bibitem{DHMS12}
V.A. Derkach, S. Hassi, M.M. Malamud, and H. de~Snoo, 
Boundary triplets and Weyl functions. Recent developments.
In {\em Operator methods for boundary value problems},
London Math. Soc. Lecture Note Ser. 404, 161--220, Cambridge Univ. Press, Cambridge, (2012).
 
 \bibitem{DM91} V.A.~Derkach and M.M.~Malamud,
Generalized resolvents and the boundary value problems for Hermitian operators with gaps,
J.\ Funct.\ Anal. 95 (1991), 1--95.

\bibitem{DM95} V.A.~Derkach and M.M.~Malamud,
The extension theory of Hermitian operators and the moment problem,
J.\ Math.\ Sci. 73 (1995), 141--242.


\bibitem{GM08}
F.~Gesztesy and M.~Mitrea,
Generalized Robin boundary conditions, Robin-to-Dirichlet maps, and
Krein-type resolvent formulas for Schr\"odinger operators on bounded Lipschitz domains,
in: Perspectives in partial differential equations,
harmonic analysis and applications.
Proc. Sympos. Pure Math., Vol.  79,  Amer. Math. Soc.,
2008, 105--173.

\bibitem{GM09} 
F.~Gesztesy and M.~Mitrea,
Robin-to-Robin maps and Krein-type resolvent formulas for
Schr\"odinger operators on bounded Lipschitz domains,
Oper. Theory Adv. Appl. 191, 81--113.



\bibitem {GM11}
F.~Gesztesy and M.~Mitrea,
A description of all self-adjoint extensions of the Laplacian
and Krein-type resolvent formulas on non-smooth domains,
J. Math. Anal.  Appl. 113 (2011), 53--172.

%

\bibitem{GG91} V.I.~Gorbachuk and M.L.~Gorbachuk, {\it Boundary Value Problems for Operator Differential Equations}, Kluwer Academic Publishers, Dordrecht, 1991.

\bibitem{G68} G. Grubb, A characterization of the non-local boundary value problems associated with an elliptic operator, Ann. Scuola Norm. Sup. Pisa 22 (1968), 425--513.
%

\bibitem{G08} G.~Grubb, Krein resolvent formulas for elliptic boundary problems in non\-smooth domains, Rend.\ Semin.\ Mat.\ Univ.\ Politec.\ Torino 66 (2008), 271--297.

\bibitem{G09}
G.~Grubb, 
Distributions and Operators, 
Graduate Texts in Mathematics 252, Springer, 2009.

\bibitem{HMM05}  S.~Hassi, M.\,M.~Malamud, and V.\,I.~Mogilevskii, Generalized resolvents and boundary triplets for dual pairs of linear relations. Methods Funct. Anal. Topology 11 (2005), 170--187.

\bibitem{HMM13}  S.~Hassi, M.\,M.~Malamud, and V.\,I.~Mogilevskii, Unitary equivalence of proper extensions of a symmetric operator and the Weyl function,
Integral Equations Operator Theory 77 (2013), 449--487.


\bibitem{K75} A.N.~Kochubei, Extensions of symmetric operators and symmetric binary relations, Math. Zametki 17 (1975), 41-48 (in Russian); translated in: Math. Notes 17 (1975), 25--28

\bibitem{LT77} H.~Langer and B.~Textorius, On generalized resolvents and $Q$-functions of symmetric linear relations (subspaces) in Hilbert space, Pacific J.\ Math.\  72 (1977), 135--165.

\bibitem{LS22} Y.~Latushkin and S.~Sukhtaiev,
Resolvent expansions for self-adjoint operators via boundary triplets,
Bull. Lond. Math. Soc. 54 (2022), 2469--2491.

\bibitem{LS23} Y.~Latushkin and S.~Sukhtaiev, First-order asymptotic perturbation theory for extensions of symmetric operators,
arXiv:2012.00247. 
%

\bibitem{LS83} V.E.~Lyantse and O.G.~Storozh, {\it Methods of the Theory of Unbounded Operators}, Naukova Dumka, Kiev, 1983.

\bibitem{M10} M.M.~Malamud, Spectral theory of elliptic operators in exterior domains, Russ.\ J.\ Math.\ Phys. 17 (2010), 96--125.

\bibitem{MM97} M.M.~Malamud and V.I.~Mogilevskii, On extensions of dual pairs of operators, Dopov. Nats. Akad. Nauk Ukr. Mat. Prirodozn. Tekh. Nauki (1997), 30--37.

\bibitem{MM99} M.M.~Malamud and V.I.~Mogilevskii, On Weyl functions and $Q$-functions of dual pairs of linear relations, Dopov. Nats. Akad. Nauk Ukr. Mat. Prirodozn. Tekh. Nauki (1999), 32--37.

\bibitem{MM02} M.M.~Malamud and V.I.~Mogilevskii, Krein type formula for canonical resolvents of dual pairs of linear relations, Methods Funct. Anal. Topology 8 (2002), 72--100.

\bibitem{MPS16} A. Mantile, A. Posilicano, and M. Sini, Self-adjoint elliptic operators with boundary conditions on not closed hypersurfaces, J. Differential Equations 261 (2016), 1--55.

\bibitem{M00} W.~McLean, \textit{Strongly Elliptic Systems and Boundary Integral Equations}, Cambridge University Press, 2000.

\bibitem{M06} V.I.~Mogilevskii, Boundary triplets and Krein type resolvent formula for symmetric operators with 
unequal defect numbers, Methods Funct. Anal. Topology 12 (2006), 258--280.

\bibitem{P04}
A.~Posilicano, 
Boundary triples and Weyl functions for singular perturbations of self-adjoint operators,
Methods Funct. Anal. Topology, 10 (2004), 57--63.

\bibitem{P08} A.~Posilicano, Self-adjoint extensions of restrictions, Operators and Matrices 2 (2008), 1--24.

\bibitem{PR09}
A.~Posilicano and L.~Raimondi, 
Kre\u{\i}n's resolvent formula for self-adjoint extensions of symmetric second-order elliptic differential operators,
J. Phys. A: Math. Theor. 42 (2009), 015204, 11pp.

\bibitem{P07} 
O.~Post, 
First-order operators and boundary triples,

Russ. J. Math. Phys. 14 (2007), 482--492.

\bibitem{P13} 
O.~Post, 
Boundary pairs associated with quadratic forms, 
Math.\ Nachr. 289 (2016), 1052--1099.

\bibitem{R07} 
V.~Ryzhov,
A general boundary value problem and its Weyl function,
Opuscula Math. 27 (2007), 305--331.

\bibitem{S12} K.~Schm\"{u}dgen,
\textit{Unbounded Self-adjoint Operators on Hilbert Space},
Graduate Texts in Mathematics 265, Springer, Dordrecht, 2012.

\bibitem{V80} L.I.~Vainerman, On extensions of closed operators in Hilbert space, Math. Notes 28 (1980), 871--875.

\bibitem{V52} M.I.~Vishik, On general boundary problems for elliptic differential operators, Trudy Moskov. Mat. Obsc. 1 (1952), 187--246 (in Russian); translated in:
Amer. Math. Soc. Transl. 24 (1963), 107--172.

\bibitem{W17}
S.A. Wegner,
Boundary triplets for skew-symmetric operators and the generation of strongly continuous semigroups,
Anal. Math. 43 (2017), 657--686.

\bibitem{W13}
R. Wietsma,
Block representations for classes of isometric operators between Krein spaces,
Oper. Matrices 7 (2013),  651--685.

\end{thebibliography}
\end{document}